\documentclass[a4paper,12pt,reqno,twoside]{amsart}

\usepackage[utf8]{inputenc} 		
\usepackage[T1]{fontenc} 

\usepackage{amssymb}
\usepackage{amsmath}
\numberwithin{equation}{section} 
\usepackage{graphicx}
\usepackage[colorlinks,citecolor=blue,urlcolor=blue]{hyperref}
\usepackage{cite}
\usepackage[hmargin=2.5cm,vmargin=2.5cm]{geometry}
\usepackage{mathrsfs} 
\usepackage{ellipsis}
\usepackage{graphicx}
\usepackage{svg}

\usepackage{dirtytalk}
\usepackage{enumitem}

\usepackage{tikz}
\usetikzlibrary{arrows.meta,positioning,calc}

\newcommand{\be}{\begin{eqnarray}}
	\newcommand{\ee}{\end{eqnarray}}

\newcommand{\beq}{\begin{equation}}
	\newcommand{\eeq}{\end{equation}}
\newcommand{\beqn}{\begin{equation*}}
	\newcommand{\eeqn}{\end{equation*}}

\newcommand{\round}[1]{\lfloor#1\rfloor}

\newtheorem{thm}{Theorem}[section]
\newtheorem{prop}[thm]{Proposition}

\newtheorem{lem}[thm]{Lemma}

\theoremstyle{remark}

\newtheorem{remark}[thm]{Remark}

\newcommand\cD{{\mathcal D}}

\newcommand\cF{{\mathcal F}}

\newcommand\cK{{\mathcal K}}

\newcommand\cU{{\mathcal U}}

\newcommand\bE{{\mathbb E}}

\newcommand\bP{{\mathbb P}}

\newcommand\bR{{\mathbb R}}

\newcommand\bZ{{\mathbb Z}}

\newcommand{\ve}{\varepsilon}

\begin{document}
	
	\title[Random walks in intermittent dynamical environments]{A local limit theorem for a random walk in an intermittent dynamical environment}
	
	\author[Juho Lepp\"anen]{Juho Lepp\"anen}
	
	\thanks{\textsc{Department of Mathematics, 
			Tokai University, Kanagawa, 259-1292, Japan}}
	\thanks{ \textit{E-mail address}: leppanen.juho.heikki.g@tokai.ac.jp}
	\thanks{\textit{Date}:  \today }
		\thanks{This research was supported by 
		JSPS via the project LEADER. The author is grateful to Mikko Stenlund for helpful discussions.}

\begin{abstract}
	We study an extended dynamical system on the non-negative real line with
	piecewise linear non-uniformly expanding local dynamics. With a uniformly distributed initial state, the distribution of successive states coincides with that of a random walk in an inhomogeneous environment. Under suitable conditions on the environment, we establish a central limit theorem and a (non-Gaussian) local limit theorem for the walk. Our approach builds on the work of Leskel\"a and Stenlund (Stochastic Process. Appl. 121(12), 2011), who analyzed a corresponding model with uniformly expanding local dynamics.
	\end{abstract}
	
	\maketitle
	
	\section{Introduction}
	
	\emph{Extended dynamical systems} are used to model complex non-equilibrium
	evolution observed in a wide range of physical phenomena, including fluid
	convection, reaction-diffusion processes, and heat conduction. 
	Although the term has no mathematically precise definition,
	statistical properties of specific models have been studied in 
	the literature of dynamical systems theory for
	several decades.
	These include coupled map lattices, which are
	discrete models of spatiotemporal chaos and have been the subject of extensive
	investigations; we refer to \cite{CF05, T23} 
	for background information and to \cite{BGK07, KL09, BP24, BS22, FGGV18, B02} for 
	results on
	statistical properties.
	Another class of examples is given by extended billiards, which describe
	transport in infinite scattering media; see, for example,
	\cite{CLS10, SLEC11, CD06, AL20}.
	Essential characteristics of extended dynamical systems include high-dimensional phase spaces
	or non-compact phase spaces with physically relevant (invariant)
	measures of infinite total mass.

	A one-dimensional extended dynamical system with strongly chaotic local dynamics 
	was analyzed in \cite{LS11}. The model can be thought of as a particle moving in an 
	inhomogeneous medium composed of a linear chain of cells.
	More precisely, the system consists of an infinite partition of the half-line 
	$\bR_+ = [0,\infty)$ into intervals $[x,x+1)$, $x \in \bZ_+$, each assigned 
	a label $\omega_x \in (0,1)$. At discrete times $n = 0,1,2,\ldots$, the position of a 
	particle moving through the cells is denoted by $u_n \in \mathbb{R}_+$. The family of labels $(\omega_x)_{x \in \bZ_+}$, 
	called an \emph{environment}, determines the local dynamics of the particle. Namely, if 
	at time $n$ the particle lies in the cell $[x,x+1)$ with index $x \in \mathbb{Z}_+$,
	i.e. $u_n \in [x,x+1)$, then at time $n + 1$ it moves to
	\begin{align}\label{eq:LS_local}
		u_{n+1} = U_{\omega_x}(u_n - x) + x,
	\end{align}
	where
	$U_{\omega_x} : [0,1) \to [0,2)$ is a piecewise linear surjective map that depends on 
	the configuration $\omega_x$ and is
	\emph{uniformly expanding}:~$\text{inf}_{z \in [0,1]} |U_{\omega_x}'(z)| > 1$.
	Choosing the initial state $u_0$ 
	uniformly 
	and randomly
	in $[0,1)$ renders $(u_n)$ into a stochastic process. 
	It was shown in \cite{LS11} that, under suitable regularity conditions on the 
	environment, $(u_n)$ satisfies a local limit theorem (LLT). 
	The form of the limit is necessarily \emph{non-Gaussian}, involving a modulating
	factor depending on the environment. This result 
	was further applied to obtain a quenched LLT for $\varphi$-mixing environments 
	satisfying appropriate moment conditions. The analysis was based on
	a reduction of the system to a one-dimensional random walk with a forbidden direction.
	
	In this work we consider the case where the local dynamics are described by
	interval maps with \emph{non-uniform} expansion. We restrict attention to
	piecewise linear maps with a countably infinite number of uniformly expanding
	branches, however with the expansion rate tending to one at the left endpoint.
	Conventional measure-preserving dynamical systems generated by such maps have been studied in \cite{W89, M93}, among others.
	
	Interval maps with neutral fixed points are often called \emph{intermittent} due to their dynamical 
	characteristics: the dynamics alternate between \emph{laminar} phases, occurring when the trajectory 
	enters a small neighborhood of a neutral fixed point, and chaotic phases, occurring due to expansion away 
	from the neutral fixed points. Such maps were introduced in the classical work of 
	Pomeau and Manneville \cite{PM80} as simple models for transitions to turbulence in convective fluids. Remarkably,
	intermittent maps exhibit a polynomial rate of correlation decay 
	\cite{LSV99, Y99} and statistical limit theorems 
	whose form depends on the strength of expansion near the neutral fixed
	points (see, e.g.,  \cite{G04, TZ06, T14, MZ15}).
	Over the past decade or so there has been substantial interest in the study of 
	extended and nonautonomous intermittent dynamical systems; works in this direction include
	\cite{BK24, KL21, BBR19, LS16, DGS24, HKRVZ22}. In these works it is observed that 
	temporal or spatial dependencies between the constituent maps play an essential role in the 
	description of 
	statistical behavior.
	
	Returning to the present study, our primary objective is to describe how
	inhomogeneities of the environment impact the long-term statistical properties
	of the intermittent extended dynamical system. Following the approach of
	\cite{LS11}, we establish a relation between the distribution of trajectories
	in the extended system and the distribution of an inhomogeneous random walk on
	$\bZ_+ = \{0,1,\ldots\}$. The tail probabilities of this random walk are
	related to the tail probabilities of the first entry times of the single-site dynamics
	into a strongly chaotic region of the state space. Such return-time statistics
	are well known to be closely connected with limit theorems in the case
	of conventional stationary dynamical systems (see \cite[Theorem 2]{Y99}). 
	We analyze statistics of the first hitting times for the random walk, and use 
	these to derive corresponding statistics for the walk.	
	In
	the case of deterministic environments, we derive regularity conditions,
	parallel to those in \cite{LS11}, under which the walk satisfies a strong law
	of large numbers, a central limit theorem, and a non-Gaussian local limit
	theorem. We verify these conditions in the case of stationary $\alpha$-mixing
	environments with suitable moment conditions, which then yield corresponding
	quenched limit theorems. We emphasize that the LLT for the walk
	is equivalent to an LLT for the intermittent extended dynamical
	system under study.
	
	There is a large body of results on central limit theorems and invariance
	principles for random walks in dynamical environments. Let us mention
	\cite{RS05, DKL08, DL08} as examples. Dolgopyat and Goldsheid \cite{DG13} proved a quenched
	LLT for ballistic one-dimensional random walks in i.i.d. environments,
	and later extended this result to walks on strips \cite{DG20}. 
	LLTs for various diffusive random walks in random environments  
	(with a focus on i.i.d. environments) were also studied in  
	\cite{BCR16, BBDS23, CD15, S13}. For measure-preserving intermittent dynamical
	systems, Gou\"{e}zel \cite{G04} proved a Gaussian LLT, recalled below in
	\eqref{eq:lclt}. In the case of Liverani--Saussol--Vaienti maps (see \eqref{eq:gen_lsv} for the definition), this theorem
	applies in the regime $\alpha < 1/2$ where
	correlations (for instance of H\"older observables) are summable. The LLT proved
	in the present paper (Theorem \ref{thm:llt}) holds under a corresponding
	restriction on the parameter range of the maps constituting the extended system.
	
	In the remainder of this section we introduce the intermittent maps that serve 
	as building blocks for the extended system, and then give a precise definition 
	of the model to be studied in the rest of the paper.
	
	\subsection{A family of intermittent maps} 
	
	Let $T : [0,1] \to [0,1]$ be defined by
	\begin{align}\label{eq:gen_lsv}
		T(x) = \begin{cases}
			x + \kappa x^{\alpha + 1}, &x \in [0, c], \\
			\frac{x - c}{1 - c}, &x \in [c,1],
		\end{cases}
	\end{align}
	where $\kappa, \alpha, c$ are parameters satisfying 
	\begin{align}\label{eq:cond_params_lsv}
	\alpha \in (0,1), \quad c + \kappa c^{\alpha + 1} = 1, \quad c \in (0,1), \quad \kappa > 0.
	\end{align}
	
	This family forms a subclass of the generalized Liverani--Saussol--Vaienti (LSV) maps studied, for example, in \cite{BL21}. Its first branch is expanding except for the neutral fixed point at the origin, while the other branches are linear and uniformly expanding. Here we restrict to the case of two branches. The classical LSV map from \cite{LSV99} is recovered by taking $\kappa = 2^\alpha$ and $c = 1/2$. 
	
	First, we recall some standard definitions and properties about the maps \eqref{eq:gen_lsv} that will be useful in the sequel.
	For all $n \ge 0$ define
	\begin{align}\label{eq:cn_lsv}
		c_n = c_n(\alpha, c, \kappa) = g^{-n}(1),
	\end{align}
	where $g = T|_{[0,c]}$.  The intervals $[c_{n+1}, c_n)$, $n \ge 1$, form a partition of $(0,c)$ 
	according to the first entry times of trajectories $T^n(x)$ 
	into the uniformly expanding region $[c,1]$. Namely,
	\begin{align}\label{eq:tau_x}
	R(x) := \inf \{  n \ge 1 \: : \: T^n(x) \in [c,1]  \}
	\end{align}
	satisfies $R |_{ [c_{n+1}, c_n) } = n$.
	The following properties are immediate from the definitions of $T$ and $c_n$:
	\begin{itemize}
		\item $c_{n+1} < c_n$ and $\lim_{n \to \infty} c_n = 0$, \smallskip 
		\item $T([c_1,1]) = [0,1]$, \smallskip 
		\item $T|_{ [ c_{i+1}, c_i ) } \in C^1$ and
		$T([c_{i+1}, c_i)) = [c_i, c_{i-1})$ for all $i \ge 1$, \smallskip 
		\item $T'(x) > 1$ for $x > 0$, and $\lim_{x \downarrow 0} T'(x) = 1$.
	\end{itemize}
	In fact, using induction as in the proof of \cite[Lemma 3.2]{LSV99}, one finds that, for $n \ge 1$,
	\begin{align}\label{eq:estim_cn}
		n^{ 1 / \alpha }c_n \le 
		\max \{
		c, \kappa^{ - 1 / \alpha } 2^{ 1/\alpha + 1 / \alpha^2 }  \}
		\le  c + c \biggl( \frac{c}{1-c} \biggr)^{1/\alpha}  
		2^{ 1/\alpha + 1 / \alpha^2 }.
	\end{align}
	Since
	$$
	c_{n-1} - c_n = T(c_n) - c_n = \kappa c_n^{\alpha + 1},
	$$
	it follows that 
	\begin{align}\label{eq:estimate_diff_cn}
		n^{ 1 / \alpha + 1 }(
		c_{n-1} - c_n) \le  (1 - c) + 
		c \biggl( \frac{c}{1-c} \biggr)^{1/\alpha}  
		2^{1 + 2/\alpha + 1 / \alpha^2 }.
	\end{align}
	For $n \ge 1$, the mean value theorem implies that
	\begin{align}\label{eq:mvt}
		T'(\xi_n) = \frac{ c_{n-1} - c_n }{ c_n - c_{n+1} } 
		\quad \text{for some $\xi_n \in (c_{n+1}, c_n)$.}
	\end{align}
	
	For the map \eqref{eq:gen_lsv}, the central limit theorem (CLT) follows from \cite{LSV99, Y99}. Specifically, let $\varphi : [0,1] \to \mathbb{R}$ be a H\"older continuous function satisfying $\int \varphi \, d\nu = 0$, where $\nu$ denotes the unique absolutely continuous invariant probability measure for $T$ (see \cite{LSV99} for the existence of such a measure).
	Then, for $\alpha < 1/2$, the partial sums  
	$
	S_n = \sum_{k=0}^{n-1} \varphi \circ T^k
	$ satisfy
	\begin{equation}\label{eq:clt_lsv}
		n^{-1/2} S_n \stackrel{\mathcal{D}}{\to} N(0, \sigma^2),
	\end{equation}  
	for some $\sigma^2 \ge 0$, where $\stackrel{\mathcal{D}}{\to}$ means convergence in distribution. Here, the sequence $(\varphi \circ T^k)$ is considered as a stochastic process on the probability space $([0,1], \nu)$. Gou\"{e}zel \cite{G05} proved a local CLT, which refines \eqref{eq:clt_lsv} by establishing convergence at the level of densities. Specifically, we have the following result.
	
	\begin{thm}[Special case of Theorem 1.2 in \cite{G05}]\label{thm:gou}
	Let $\alpha < 1/2$.
	Suppose that $\varphi$ cannot be written as 
	$\varphi = c + \psi - \psi \circ T + \lambda q$ 
	for $\lambda > 0$, $c \in \bR$, and some measurable functions $q : [0,1] \to \bZ$, 
	$\psi : [0,1] \to \bR$.
	Then, for any bounded interval $J \subset \bR$,
	and any real sequence $(k_n)$ with $k_n / \sqrt{n} \to \kappa \in \bR$,
	\begin{align}\label{eq:lclt}
		\sqrt{n} \nu\biggl( 
		S_n \in J + k_n
		\biggr) = |J| \phi_{0,  \sigma^2 }(\kappa) + o(1).
	\end{align}
	Here, $\phi_{\mu, \sigma^2}$ denotes the density of $N(\mu, \sigma^2)$:
	\begin{align}\label{eq:dens_normal}
	\phi_{\mu, \sigma^2}(z) = (2 \pi \sigma^2 )^{-1/2} 
		\exp\biggl( - \frac{(z - \mu)^2}{2  \sigma^2 }  \biggr).
	\end{align}
	\end{thm}
	
	\subsection{Model definition}
	
	\subsubsection{A piecewise linear intermittent extended dynamical system}
	Let $\tilde{c} = (c_n)_{n \ge 1}$ be a sequence of real numbers such that 
	$0 < c_{n+1} < c_n < 1$ for all $n \ge 1$ and $c_n \to 0$ as $n \to \infty$.
	Motivated by \eqref{eq:mvt}, we define a piecewise linear analog
	of the map \eqref{eq:gen_lsv} as follows:
	\begin{align}\label{eq:pw_linear_lsv}
		\tilde{T}_{ \tilde{c} }(x)
		= \begin{cases}
			\frac{x - c_1}{1 - c_1}, &x \in [c_1, 1], \\
			\frac{c_{n-1} - c_n}{c_n - c_{n+1}} ( x - c_{n+1} ) + c_n, 
			&x \in [c_{n+1}, c_n), \: n \ge 1.
		\end{cases}
	\end{align}
	We also set $\tilde{T}_{\tilde{c}}(0)=0$. 
	In other words, $\tilde{T}_{\tilde{c}}$ maps $[c_1,1]$ linearly onto $[0,1]$ and, for each $n \ge 1$, maps
	$[c_{n+1}, c_n)$ linearly onto $[c_n, c_{n-1})$, with
	$$
	( \tilde{T}_{\tilde{c}}|_{(c_{n+1}, c_n)} )'(x) 
	= \frac{ c_{n-1} - c_n }{ c_n - c_{n+1} }, \quad n \ge 1, 
	\qquad ( \tilde{T}_{\tilde{c}}|_{(c_{1}, 1)} )'(x) = \frac{1}{1 - c_1}.
	$$
	Note that,
	in the special case where the sequence $(c_n)$ is defined by \eqref{eq:cn_lsv}, 
	the first entry time function of $\tilde{T}_{\tilde{c}}$ coincides with \eqref{eq:tau_x}.

	We consider a variant of the extended dynamical system of \cite{LS11}, 
	with local dynamics described by maps of the form \eqref{eq:pw_linear_lsv}.
	A sequence 
	$$\omega = (\omega^x)_{x \ge 0} \in \bigl( (0,1)^{\bZ_+} \bigr)^{\bZ_+}$$
	is called an \emph{environment} if, for each $x \ge 0$, 
	$\omega^x = (\omega_n^x)_{n \ge 0} \in (0,1)^{\bZ_+}$ satisfies
	\begin{align}\label{eq:environment}
		\lim_{n \to \infty} \omega^x_n = 0, 
		\qquad 
		\omega_{n+1}^x < \omega_n^x \le \omega_0^x = 1 
		\quad \forall n \ge 0.
	\end{align}
	Given such an environment $\omega$, for each $x \ge 0$ we define a map $U_x : [0,1) \to [0,2)$ by
	\begin{align}\label{eq:loc_dyn}
		U_x(u)
		= \begin{cases}
			\tilde{T}_{ \omega^x }(u), &u \in [0, \omega_1^x), \\
			\tilde{T}_{\omega^x}(u) + 1, &u \in [\omega_1^x, 1),
		\end{cases}
	\end{align}
	where $\tilde{T}_{\omega^x}$ is defined as in \eqref{eq:pw_linear_lsv}. 
	Thus $U_x$ is a piecewise linear transformation that maps 
	$[\omega_{n+1}^x, \omega_n^x)$ onto $[\omega^x_n, \omega^x_{n-1})$ for $n \ge 1$, and maps
	$[\omega_1^x, \omega_0^x) = [\omega_1^x, 1)$ onto $[1,2)$.
	
	Analogously to \eqref{eq:LS_local},
	the local maps $U_x$ describe the motion of
	a particle in an inhomogeneous medium formed by a chain of infinitely many
	cells, represented by the subintervals
	$[x, x+1)$, $x \in \bZ_+$, of the nonnegative real line $\bR_+$. 
	If the position of the particle at time $n$ is $u_n \in [x, x+1)$, 
	then at time $n+1$ its position is given by
	$$
	u_{n+1} = x + U_x(u_n - x).
	$$
	Equivalently,
	\begin{align}\label{eq:chaos}
		u_n = \cU_\omega^n(u_0),
	\end{align}
	where 
	the global map $\cU_\omega : \bR_+ \to \bR_+$ is defined by
	$$
	\cU_\omega(u) = \lfloor u \rfloor + U_{\lfloor u \rfloor}(u - \lfloor u \rfloor),
	$$
	with $\lfloor u \rfloor = \max \{ n \in \bZ : n \le u \}$ denoting the integer part of $u$. 
	We assume throughout that the initial state of the particle $u_0$ is located in the leftmost cell $[0,1)$.
		
	\subsubsection{Random walk from intermittent dynamics}\label{sec:rwre}
	
	Given an environment $\omega = (\omega^x)_{x \ge 0}$ satisfying \eqref{eq:environment}, we 
	define a transition kernel on $\bZ_+ \times \bZ_+$
	as follows (see Figure \ref{fig:chain}):	
	\begin{align*}
		\cK( (x,y), (x', y') ) = 
		\begin{cases}
			1, &\text{if $y' > 0$ and $y = y' - 1$ and $x = x'$}, \\
			1 - \omega_1^x, &\text{if $y = y' = 0$ and $x = x' + 1$}, \\
			\omega_{y }^{  x } - \omega_{y + 1}^{  x}, &\text{if $y'=0$, $y > 0$, and $x = x' + 1$}, \\
			0, &\text{otherwise.}
		\end{cases}
	\end{align*}	
	
\begin{figure}[h]
\begin{tikzpicture}[
	scale=0.85,
	x=5.0cm, y=1.90cm, >=Latex,
	state/.style={circle, draw, minimum size=5.5mm, 
	inner sep=1pt, font=\scriptsize},
	lbl/.style={font=\scriptsize, 
	inner sep=1pt, fill=white, fill opacity=1, text opacity=1}
	]
	
	\node at (0,2.9) {\small $x=0$};
	\node at (1,2.9) {\small $x=1$};
	\node at (2,2.9) {\small $x=2$};
	
	\node at (0,2.6) {$\vdots$};
	\node at (1,2.6) {$\vdots$};
	\node at (2,2.6) {$\vdots$};
	
	\foreach \x/\col in {0/A,1/B,2/C}{
		\foreach \y in {0,1,2}{
			\node[state] (\col\y) at (\x,\y) {$(\x,\y)$};}
	}
	
	\foreach \col in {A,B,C}{
		\draw[->] (\col2) -- node[lbl,right,xshift=2pt]{1} (\col1);
		\draw[->] (\col1) -- node[lbl,right,xshift=2pt]{1} (\col0);
	}
	
	\draw[->, bend left=15] (A0) to node[lbl,above,sloped, yshift=2pt, pos=0.52] {$1-\omega^{1}_{1}$} (B0);
	\draw[->, bend left=5]  (A0) to node[lbl,above,sloped, yshift=2pt, pos=0.70] {$\omega^{1}_{1}-\omega^{1}_{2}$} (B1);
	\draw[->, bend right=5] (A0) to node[lbl,above,sloped, yshift=2pt, pos=0.55] {$\omega^{1}_{2}-\omega^{1}_{3}$} (B2);
	
	\draw[->, bend left=15] (B0) to node[lbl,above,sloped, yshift=2pt, pos=0.52] {$1-\omega^{2}_{1}$} (C0);
	\draw[->, bend left=5]  (B0) to node[lbl,above,sloped, yshift=2pt, pos=0.70] {$\omega^{2}_{1}-\omega^{2}_{2}$} (C1);
	\draw[->, bend right=5] (B0) to node[lbl,above,sloped, yshift=3pt,  pos=0.55] {$\omega^{2}_{2}-\omega^{2}_{3}$} (C2);
	
	\foreach \y in {0,1,2}{
		\node at (2.5,\y) {$\cdots$};
	}
	
\end{tikzpicture}
\caption{Transition kernel $\cK$.}\label{fig:chain}
\end{figure}

Let $(Z_n)_{n \ge 0}$ be a Markov chain on $\bZ_+ \times \bZ_+$ with initial distribution and transition probabilities
as follows:
\begin{align}\label{eq:walk}
	\begin{split}
	&P( Z_0 = (0,y) ) = 	\omega_{y}^0 - \omega_{y+1}^0, \quad \forall y \in \bZ_+, \\
	&P( Z_{n+1} =  (x,y) \: | \: Z_n = (x', y')    ) = 	\cK( (x,y), (x', y') ).
	\end{split}
\end{align}
After fixing the environment $\omega$, we denote by $P_\omega$ the distribution
of the Markov chain $(Z_n)$ on the path space $(\bZ_+\times\bZ_+)^{\bZ_+}$.
We write $E_\omega$ and $\text{Var}_\omega$ respectively for the expectation and variance with respect to $P_\omega$.

	For $x,y \in \bZ_+$, 
	let $I_{(x ,y)} \subset \bR$ denote the subinterval  corresponding to the state $(x,y)$ of $(Z_n)$:
	$$
	I_{(x,y)} = [ x + \omega^x_{y+1}, x +  \omega^x_{y } ).
	$$
	We also define $I_{(-1,y)} = \emptyset$.
	
	\begin{prop} Suppose that $u_0 \sim \text{Uniform}[0,1)$. Then, for all $n \ge 0$, $a,b \in \bR$,
		\begin{align}\label{eq:distr_un}
			P( u_n \in (a,b)  \: | \:  u_n \in I_{ (x,y)} ) = \frac{|(a,b)  \cap I_{(x,y)}|}{ | I_{(x,y)} | },
		\end{align}
		provided that $P( u_n \in I_{(x,y)} ) > 0$.
	\end{prop}
	
	\begin{proof} We proceed by induction on $n$. For $n = 0$, the claim follows immediately from the assumption that $u_0 \sim \text{Uniform}[0,1)$. Assume that \eqref{eq:distr_un} holds for $n$. Suppose that 
		$P( u_{n+1} \in I_{(x,y)} ) > 0$, and without loss of generality that
		$$
		x + \omega_{y+1}^x \le a \le b \le x + \omega_{y}^x,
		$$
		so that $(a,b) \subset I_{(x,y)}$.
		By the definition of $\cU_\omega$, 
		$$
		\{  u_{n+1} \in  I_{(x,y)}  \} = \{ u_n \in I_{ (x, y+1) }  \} \cup \{  u_n \in I_{(x-1, 0)}, \, u_{n+1} \in  I_{(x,y)}   \}.
		$$
		Therefore, at least one of $p_1 := P_\omega( u_n \in I_{ (x, y+1) }   )$ and $p_2 := P_\omega( u_n \in I_{(x-1, 0)} )$ is positive.
		Further, for some $C_{i,n} > 0$ independent of $a,b$:
		\begin{align*}
			&P(u_{n+1} \in (a,b) \: | \: u_{n+1} \in I_{(x,y)} ) \\
			&=  \mathbf{1}_{p_1 > 0} C_{1,n} P(  \cU_\omega (u_n) \in (a,b) \: | \: u_n \in I_{ (x,y+1) } ) 
			+ \mathbf{1}_{p_2 > 0}   C_{2,n} P(  \cU_\omega (u_n) \in (a,b) \: | \: u_n \in I_{ (x-1,0) } ).
		\end{align*}
		Now, by a direct computation using the induction hypothesis and the definition of $\cU_\omega$, we find that, 
		for some $C_{i,n}' > 0$ independent of $a,b$:
		\begin{align*}
			&P(  \cU_\omega (u_n) \in (a,b) \: | \: u_n \in I_{ (x,y+1) } ) = C_{1,n}' (b-a), \quad \text{if $p_1$ > 0, and} \\
			&P(  \cU_\omega (u_n) \in (a,b) \: | \: u_n \in I_{ (x-1,0) } ) = C_{2,n}' (b-a), \quad \text{if $p_2 > 0$.}
		\end{align*}
	\end{proof}
	
	Write $Z_n = (X_n,Y_n)$ for the components of the Markov chain. 
	The following result shows that the distribution of the trajectory 
	$\cU_\omega^n(u_0)$ generated by \eqref{eq:chaos} 
	is determined by the distribution of the horizontal coordinate $X_n$.

	\begin{prop}\label{prop:walk_ds_rel} Suppose that $u_0 \sim \text{Uniform}[0,1)$.
		Then, for any $n \ge 0$,
		\begin{align}\label{eq:relation-1}
			P(  u_n \in I_{ (x,y) } ) = P_{\omega}(  Z_n = (x,y)  ).
		\end{align}
		In particular,
		\begin{align}\label{eq:relation-2}
			P( u_n \in [x, x + 1) ) = P_{\omega}(  X_n = x  ).
		\end{align}
	\end{prop}
	
	\begin{proof}
		We prove \eqref{eq:relation-1} by induction on $n$. For $n=0$ it holds by
		\eqref{eq:walk} and the assumption that $u_0\sim\text{Uniform}[0,1)$. Assume it holds for $n$, and
		let $(x,y)\in\bZ_+\times\bZ_+$. Using the induction hypothesis and the definition of
		$\cU_\omega$,
		\begin{align*}
			P( u_{n+1} \in I_{ (x,y) } ) &= P(u_n \in I_{(x, y + 1)} ) + P( \cU_\omega(u_n) \in  I_{ (x,y) } \: | \: 
			u_n  \in I_{(x-1,0)} ) P( u_n  \in I_{(x-1,0)}  ) \\
			&= P_\omega(  Z_n = (x, y + 1) ) \\
			&+ P( \cU_\omega(u_n) \in  I_{ (x,y) } \: | \: 
			u_n  \in I_{(x-1,0)} ) 
			P_\omega( Z_n  = (x-1, 0) ),
		\end{align*}
		where the second term vanishes if $x=0$. On the other hand, by the transition rules of
		$(Z_n)$,
		\begin{align*}
			P_\omega(Z_{n+1} = (x,y)) &= P_\omega(Z_n = (x, y + 1)) 
			\\
			&+  P_\omega( Z_{n+1} = (x,y) | Z_n = (x-1,0) ) 
			 P_\omega(Z_n = (x-1, 0)) \\
			&= P_\omega(Z_n = (x, y + 1)) +  (  \omega^x_{y} - \omega^x_{y+1}  ) P_\omega(Z_n = (x-1, 0)).
		\end{align*}
		Thus, to prove \eqref{eq:relation-1} it suffices to check that
		$$
		P_\omega\biggl( \cU_\omega(u_n)\in I_{(x,y)} \, |\, u_n\in I_{(x-1,0)}
		\biggr)
		= \omega_y^x-\omega_{y+1}^x,
		$$
		whenever $P(u_n\in I_{(x-1,0)})>0$. This follows directly from \eqref{eq:distr_un}
		and the definition of $\cU_\omega$. Finally, to obtain \eqref{eq:relation-2}, note that
		\[
		P(u_n\in [x,x+1))=\sum_{y=0}^\infty P(u_n\in I_{(x,y)})
		=\sum_{y=0}^\infty P_\omega(Z_n=(x,y)) = P_\omega(X_n=x).
		\]
	\end{proof}
	
	\subsection*{How the paper is organized} The organization of the paper is simple: Section \ref{sec:results} contains the main results on limit theorems for the random walk, both in deterministic and in random 
	environments. Section \ref{sec:proof} is devoted to the proofs of these results.
	
	\section{Statement of main results}\label{sec:results}
	
	In this section we present our main results on the asymptotic behavior of the 
	random walk $(X_n)$ defined as the horizontal coordinate of the Markov chain 
	\eqref{eq:walk}.
	We begin by defining some basic notions that will be used in the analysis of the walk, and then formulate conditions under which a strong law of large numbers, a central limit theorem, and a local limit theorem hold. These results are established both for deterministic regularly varying environments (Section~\ref{sec:regular_limit}) and for weakly dependent
	random environments (Section~\ref{sec:random_limit}).

	\noindent\textbf{Convention.} In the sequel we often use $x$ to denote spatial variables and $n$ 
	to denote temporal variables.
	For sequences $a_{x,n}, b_{x,n} \in \bR$, we write $a_{x,n} = O(b_{x,n})$ if there exists a constant $C > 0$ independent of both $n$ and $x$, such that $|a_{x,n}| \le C\, b_{x,n}$ for all $x,n$.  
	For two sequences $s_{1,x}, s_{2,x}$ depending on $x$, we write $s_{1,x} = o(s_{2,x})$ if
	$
	\lim_{x \to \infty} s_{1,x} / s_{2,x} = 0.
	$
	Similarly, 
	for two sequences $r_{1,n}, r_{2,n}$ depending on $n$, we write $r_{1,n} = o(r_{2,n})$ if
	$
	\lim_{n \to \infty} r_{1,n} / r_{2,n} = 0.
	$

  \subsection{Definitions}   Let $\beta : \bZ_+ \to (1, \infty)$ be a function such that
  \begin{align}\label{eq:beta_star}
  	\beta_* := \inf_{ x \in \bZ_+ } \beta(x) > 1,
  \end{align}
  and define 
  \begin{align}\label{eq:A_x}
  	A_x = \biggl( \sup_{n \ge 1}  n^{ \beta(x) } \omega_n^x \biggr) \vee 1, \quad A_x' = \biggl( \sup_{n \ge 1}  n^{ \beta(x) + 1 } ( \omega_{n-1}^x - \omega_n^x  ) \biggr)  \vee 1.
  \end{align}
  Note that
  \begin{align}\label{eq:rel_ax_ax'}
  	A_x \le (A_x' ( 1 + \beta_*^{-1} )) \vee 1.
  \end{align}
  
  For each $x \in \bZ_+$, define
  \begin{align}\label{eq:T_tau}
  T_x = \inf \{  n \ge 0 \: : \: X_n = x \} \quad \text{and} \quad 
  \tau_x = T_{x + 1} - T_x.
  \end{align}
  Thus $T_x$ is the first hitting time of $x$ and $\tau_x$ is the sojourn time at $x$ for the walk $X_n$. A useful property obvious from the definitions is that
  $
  T_x = \sum_{w = 0}^{x-1} \tau_{w}.
  $
  
  Set
  \begin{align*}
  	m_x =  \sum_{n=0}^\infty \omega_{n}^x \quad
  	\text{and} \quad 
  	 \quad \mu_x = \sum_{w=0}^{x-1} m_w,
  \end{align*}
  which are the mean of $\tau_x$ and $T_x$, respectively (see \eqref{eq:m_x}).  
  Finally, let
  \begin{align}\label{eq:Mn}
  M_n = \inf \{ x \ge 0 \: : \: \mu_x \ge n  \}, \quad n \ge 0,
  \end{align}
  which is a generalized inverse of the sequence $(\mu_x)$. 
  
  \subsection{Limit theorems for deterministic environments}\label{sec:regular_limit}
  Fix an environment $(\omega^x)$ satisfying \eqref{eq:environment}.  
  We impose the following assumption which requires that 
  averages of the expected sojourn times 
  converge to a limit $\mu > 1$. Here $\mu^{-1}$ may be interpreted as the traveling 
  speed of the particle.
  
  \noindent\textbf{Assumption (A1).} 
  There exists $\mu>1$ such that
  $$
  \frac{1}{x}\sum_{w=0}^{x-1} m_w \;=\; \mu + \theta_1(x),
  \qquad \theta_1(x)=o(1).
  $$
  
  \begin{thm}[Strong law of large numbers]\label{thm:slln} Assume that
  	\begin{align}\label{eq:cond_ax_slln}
  	A_x = O( x^\lambda )
  	\end{align}
  	for some $0 \le \lambda < \min\{ 1, \beta_* - 1 \}$, and that (A1) holds with arbitary
  	$\theta_1(x) = o(1)$.
  	Then,
  	\begin{align}\label{eq:slln}
  	\lim_{ n \to \infty }  n^{-1} X_n = \mu^{-1} \quad \text{a.s.}
  	\end{align}
  \end{thm}
  
\begin{remark}
	Consider the special case
	\begin{align}\label{eq:lsv_environment}
		\omega_n^x = g_x^{-n}(1),
	\end{align}
	where $g_x = T_x|_{[0,c_x]}$ and $T_x$ is the map in \eqref{eq:gen_lsv} with parameters
	$$
	\alpha_x \in (0,1/2), \qquad c_x \in (0,1), \qquad \kappa_x > 0
	$$
	satisfying \eqref{eq:cond_params_lsv}.  
	If we set $\beta(x) = 1/\alpha_x$, then the condition $\beta_* > 1$ is equivalent to 
	$\inf_{x \ge 0} \alpha_x < 1$.  
	Moreover, from \eqref{eq:estim_cn} we deduce that \eqref{eq:cond_ax_slln} holds provided that
	\begin{align}\label{eq:cond_alpha_slln}
		\limsup_{x \to \infty} \frac{1}{ \log(x) } \biggl\{ 
		\log(c_x) + \frac{1}{ \alpha_x } \log \biggl( \frac{c_x}{1-c_x} \biggr)
		+ \biggl( \frac{1}{ \alpha_x } + \frac{1}{ \alpha_x^2 } \biggr) \log(2)
		\biggr\} \le \lambda 
	\end{align}
	for some $0 \le \lambda < \min\{ 1, \beta_* - 1 \}$.
	For instance, if $\sup_{x \ge 0} c_x < 1$, then \eqref{eq:cond_alpha_slln} holds if
	$
	\alpha_x^{-1} = o( \log^{1/2}(x)  ).
	$
\end{remark}
  
  If $\beta_* > 2$, 
  we define
  $$
   \sigma_x^2 = \sum_{w=0}^{x-1} \biggl\{ 
   \sum_{n=0}^\infty (2n + 1) \omega_n^w - \biggl[
   \sum_{n=0}^\infty 
   \omega_n^w \biggr]^2,
  \biggr\},
  $$
  which equals the variance of $T_x$ (see \eqref{eq:sigma}) and is finite whenever 
  $A_w < \infty$ for $0 \le w < x$.

  \noindent\textbf{Assumption (A2).} There exists
  $\sigma^2 > 0$ such that
  \begin{align*}
  	x^{-1} \sigma_x^2 = \sigma^2 + \theta_2(x), \qquad \theta_2(x) = o(1).
  \end{align*}
  
  \begin{thm}[Central limit theorem]\label{thm:clt} Let $\beta_* > 2$. Assume that
  	\begin{align}\label{eq:ax_ord_clt}
  		A_x = o( \sqrt{x} ) \quad \text{and} \quad 
  		\sum_{w = 0}^{x-1} A_w^2 = o(x^{\beta_* / 2}),
  	\end{align}
  	and that (A1) and (A2) hold with arbitrary $\theta_i(x) = o(1)$, $i = 1,2$.
  	Then,
  	\begin{align}\label{eq:clt}
  		\frac{ X_n - n \mu^{-1} }{ \sqrt{n} \tilde{\sigma} } \stackrel{\cD}{\to} N(0,1), \quad \text{as $n \to \infty$},
  	\end{align}
  	where $\stackrel{\cD}{\to}$ means convergence in distribution and 
  	$\tilde{\sigma}^2 = \sigma^2 / \mu^3$.
  \end{thm}
  
  \begin{remark} 
  	In the special case of environments defined through \eqref{eq:lsv_environment},
\eqref{eq:ax_ord_clt} holds provided that, as $x \to \infty$,
\begin{align*}
	&-2^{-1} \log(x) 
	+ 	\log(c_x) + \frac{1}{ \alpha_x } \log \biggl( \frac{c_x}{1-c_x} \biggr)
	+ \biggl( \frac{1}{ \alpha_x } + \frac{1}{ \alpha_x^2 } \biggr) \log(2)
	\to  -\infty, \\
	&x^{ - \beta_* / 2 } \sum_{w = 0}^{ x -1  } 
	c_w^2 \biggl( \frac{c_w}{1-c_w} \biggr)^{2/\alpha_w}  
	2^{ 2/\alpha_w + 2 / \alpha_w^2 }
	= o(1).
\end{align*}
For example, if $\sup_x c_x < 1$, then \eqref{eq:ax_ord_clt} is satisfied provided that
$
\alpha_x^{-1} = o( \log^{1/2}(x)  ).
$
  \end{remark}
  
  Set
  \begin{align}\label{eq:L_def}
  L(x; u) = 	\max_{\ell : |\ell| \le u b(x)} \biggl| 
  \sum_{ w \in [ x ]_\ell } ( m_w  - \mu )
  \biggr|, \quad b(x) = (x \log x)^{1/2},
  \end{align}
  where $[x]_\ell = [x, x + \ell - 1] \cap \bZ_+$ if $\ell \ge 0$ and 
  $[x]_\ell = [x + \ell - 1, x] \cap \bZ_+$ if $\ell < 0$.

The following result establishes a local limit theorem that characterizes the asymptotic behavior of  
$P_\omega(X_n = x)$, or equivalently (via \eqref{eq:relation-2}),  
$P(\mathcal{U}^n_\omega(u_0) \in [x, x+1))$.  
Unlike the local CLT \eqref{eq:lclt}, the limiting distribution here is a weighted sum of Gaussian densities, where the weights are described by the environment configurations $\omega_n^x$.

  \begin{thm}[Local limit theorem]\label{thm:llt} Let $\beta_* > 2$. Assume that the following hold
  for some $\ve > 0$ and 
  $$
  0 \le \eta < \min \biggl\{  \frac12, \frac{\beta_* - 2}{2}   \biggr\}.
  $$
  \begin{itemize}
  	\item[(B1)] Assumptions (A1) and (A2) hold with
  	\begin{align}\label{eq:theta_ass}
  		\theta_i(x) = o( x^{-\eta}  ( \log x )^{-1/2} ), \quad i = 1,2.
  	\end{align}
  	\item[(B2)] Define
  	  	 \begin{align*}
  		B_w = \begin{cases}
  			A_w', &\text{if $\beta_* \in (2,3]$,} \\
  			A_w, &\text{if $\beta_* > 3$.}
  		\end{cases}
  	\end{align*}
  	Then
  	\begin{align}\label{eq:ass_A}
  		A_x = O(x^\eta) \quad \text{and} \quad x^{-1} \sum_{w=0}^{x-1} B_w^{3 + \ve} = O(1).
  	\end{align}
  	\item[(B3)] For all $u > 0$, 
  	\begin{align}\label{eq:L_ass}
  		L(x; u) = o(x^{1/2 - \eta}).
  	\end{align}
  	\item[(B4)] The quantity
  	  \begin{align}\label{eq:aperiod-main}
  		K_x = \sup_{n \ge 2} \frac{  \omega_{n}^x - \omega_{n+1}^x }{ \omega_1^x - \omega_2^x }
  		+ \frac{ \omega_1^x - \omega_2^x  }{ 1 - \omega_1^x }
  	\end{align}
  	satisfies $x^{-1} \sum_{w = 0}^{x-1} K_w = O(1)$.
  \end{itemize}
  Let $h_\ell$ be the density of $N(M_\ell, \ell \tilde{\sigma}^2)$, 
  where $M_\ell$ is defined by \eqref{eq:Mn}, and let $\tilde{\sigma}^2 = \sigma^2 / \mu^3$.
  Then, as $n \to \infty$, 
  \begin{align}\label{eq:local_limit}
  	P_\omega(X_n = x) = \mu^{-1} \sum_{\ell = 1}^n h_\ell(x) \omega_{n - \ell}^x +
  	o(n^{-1/2}),
  \end{align}
  where the error term is uniform in $x$.
  \end{thm}

  \begin{remark}
  In the case of environments defined by \eqref{eq:lsv_environment} we have 
  $\omega_{n}^x - \omega_{n+1}^x \le \omega_{n-1}^x - \omega_n^x$ due to expansion. 
  Hence, (B4) holds automatically with $K_x \le 2$. Further, from \eqref{eq:estim_cn} and \eqref{eq:estimate_diff_cn} we see that (B2) holds if:
  \begin{align*}
  	&\limsup_{x \to \infty} \frac{1}{ \log(x) } \biggl\{ 
  	\log(c_x) + \frac{1}{ \alpha_x } \log \biggl( \frac{c_x}{1-c_x} \biggr)
  	+ \biggl( \frac{1}{ \alpha_x } + \frac{1}{ \alpha_x^2 } \biggr) \log(2)
  	\biggr\} < \eta, \\
  	&\frac{1}{x} \sum_{w=0}^{x-1}
  	c_w^{3 + \ve} \biggl( \frac{c_w}{1-c_w} \biggr)^{(3 + \ve)/\alpha_w}  
  	2^{2(3 + \ve)/\alpha_w + (3 + \ve) / \alpha_w^2 }
  	  = O(1).
  \end{align*}
  In particular, if $\alpha_x \in [\alpha_{-}, \alpha_{+}]$ with $0 < \alpha_{-} < \alpha_{+} < 1/2$ and $\sup_{x} c_x < 1$, then (B2) holds for any $\eta > 0$ and $\ve > 0$. If, in addition, (B1) and (B3) hold 
  for some $\eta > 0$, 
  then \eqref{eq:local_limit} holds.
  \end{remark}
  
  \begin{remark}
  Assumptions (B1) and (B3) correspond to \cite[equations (2.5), (2.6)]{LS11} and \cite[equation (2.8)]{LS11},
  respecively, in the uniformly expanding model of \cite{LS11}.
  \end{remark}

  \subsection{Quenched limit theorems for random environments}\label{sec:random_limit}
  We now turn to random environments, in which the deterministic parameters $\omega_x^n$
  from Section~\ref{sec:regular_limit} are replaced by 
  realizations of suitable
  random variables.
  
  More precisely,
  let $(\Omega, \cF, \bP)$ be an arbitrary probability space equipped with a 
  measurable 
  transformation $\Theta : \Omega \to \Omega$ that preserves $\bP$:
  $
  \Theta_* \bP  = \bP.
  $
  Let $\gamma : \Omega \to \Omega_0$ 
  be a measurable transformation, where 
  $$
  \Omega_0 = \{ (\omega_n)_{n \ge 0} \in (0,1]^{\bZ_+}  \: : \:  \text{$\omega_0 = 1$ and
  $\omega_{n} > \omega_{n+1} \to 0$
  }  \}.
  $$
  We equip $\Omega_0$
  with the sigma algebra of Borel sets. We then define the environment at site $x$ by
  \begin{align}\label{eq:random_env}
  (\omega^x_n)_{n \ge 0} = 
  \gamma(\Theta^x(\omega))
  \end{align}
  where $\omega$ is sampled according to $\bP$. Thus $\Theta$ \say{shifts} the environment along the sites, and $\gamma$ specifies
  how each $\Theta^x(\omega)$ generates the local sequence $(\omega^x_n)_{n\ge0}$.
  
  Let $\bar{\beta} : \Omega \to (1, \infty)$ be a function such that 
  $$
  \bar{\beta}_* := \text{ess inf}_{\omega \in \Omega} \bar{\beta}(\omega) > 1,
  $$
  and define
  $$
  \bar{A}(\omega) = \biggl(   \sup_{n \ge 1}  n^{ \bar{\beta}(\omega) } \gamma( \omega )_n  \biggr), \quad 
  \bar{A}'(\omega) = \biggl(   \sup_{n \ge 1}  n^{ \bar{\beta}(\omega) + 1 } ( \gamma( \omega )_{n-1}  - \gamma( \omega )_n ) \biggr).
  $$
  These quantities play the roles of $\beta_*$, $A_x$, and $A_x'$ in the random setting.
  
  \begin{thm}[Quenched SLLN and CLT]\label{thm:clt_random} Let $(\Theta, \bP)$ be ergodic.
  \begin{itemize}
  	\item[(i)] Assume that $\bar{\beta}_* > 1$ and that
  	  \begin{align}\label{eq:A_moment_slln}
  		\bE[\bar{A}^q] < \infty	
  	\end{align}
  	holds with $q > 1 / \min \{ \bar{\beta}_* - 1 , 1 \}$. Then the strong law of large numbers
  	 \eqref{eq:slln} holds for $\bP$-a.e. $\omega \in \Omega$. \smallskip
  	\item[(ii)] Assume that  $\bar{\beta}_* > 2$, and that 
  	  \begin{align}\label{eq:A_moment_clt}
  		\bE[\bar{A}^2] < \infty.
  	\end{align}
  	Then the central limit theorem
  	\eqref{eq:clt} holds for $\bP$-a.e. $\omega \in \Omega$.
  	  \end{itemize}
  	\end{thm}

  For the local limit theorem in the random setting we require sufficiently fast polynomial 
  rate of $\alpha$-mixing.
  Define the $\alpha$-mixing coefficients by
  $$
  \alpha(k) = \sup \{  | \bP(A \cap B) - \bP(A) \bP(A) |  \: : \: A \in \cF_{ 0 }^\ell, \: B \in \cF_{\ell + k}^\infty, \: \ell \ge 0
   \},
  $$
  where $\cF_0^\ell = \sigma(  \gamma \circ \Theta^j  \: : \: 0 \le j \le \ell  )$ 
  and $\cF_{\ell + k}^\infty = \sigma(  \gamma \circ \Theta^j  \: : \:  j \ge \ell + k  )$.

  \begin{thm}[Quenched LLT]\label{thm:llt_random} 
  	Let $\bar{\beta}_* > 2$. Assume that the following conditions hold 
  for 
  $$
  q > \max \biggl\{   8,  \frac{2}{ \bar{\beta}_* - 2}   \biggr\}.
  $$
  \begin{itemize}
  	\item[(C1)] $\bE[ \bar{B}^q ] < \infty$
  	where 
  	$$
  	\bar{B} = \begin{cases}
  		\bar{A}, &\text{if $\bar{\beta}_* > 3$} \\
  		\bar{A}', &\text{if $\bar{\beta}_* \in (2,3]$}.
  	\end{cases}
  	$$
  	\item[(C2)] The $\alpha$-mixing coefficients satisfy
  	$$
  	\alpha(n) = O( n^{ - v }  ) \quad \text{with $v > \frac{2q}{q - 8}$}.
  	$$
  	\item[(C3)] The random variable
  	$$
  	\bar{K}(\omega) = \sup_{n \ge 2} \frac{   \gamma(\omega)_n   - \gamma(\omega)_{n+1}  }{  \gamma(\omega)_1 - \gamma(\omega)_2  }
  	+ \frac{  \gamma(\omega)_1  - \gamma(\omega)_2 }{ 1 -  \gamma(\omega)_1 } 
  	$$
  	satisfies $\bE[ \bar{K}^{1 + \delta} ] < \infty$ for some $\delta > 0$.
  \end{itemize}
  Define $h_\ell$ as in Theorem \ref{thm:llt}.
  Then, as $n \to \infty$,
   \begin{align*}
  	P_\omega(X_n = x) = \mu^{-1} \sum_{\ell = 1}^n h_\ell(x) \omega_{n - \ell}^x + 
  	o(n^{-1/2}),
  \end{align*}
  where the error term is uniform in $x$.
  \end{thm}
  
  \begin{remark}
  	In environments of the form \eqref{eq:lsv_environment}, where 
  	$(\alpha_x,\kappa_x,c_x)$ is a realization of a stationary $\alpha$-mixing process 
  	with coefficients decaying as in (C2), corresponding quenched SLLNs, CLTs, and LLTs can be formulated. 
  	In this case, (C3) is automatically satisfied due to expansion, and the moment conditions on 
  	$\bar A$ and $\bar B$ can be expressed in terms of the random parameters 
  	$\alpha_x, c_x$ using \eqref{eq:estim_cn} and \eqref{eq:estimate_diff_cn}.
  \end{remark}

\section{Proofs}\label{sec:proof}

Throughout this section, $C$ denotes a generic constant independent of both 
the temporal variable $n$ and the spatial variable $x$. 
The value of $C$ may change from one appearance to the next. 
We also continue to use the convention introduced at the beginning of 
Section~\ref{sec:results}.

  Recall the definitions of $T_x$ and $\tau_x$ from \eqref{eq:T_tau}. From the definition of 
  $Z_n = (X_n, Y_n)$ in \eqref{eq:walk}, it is clear that, for $x \ge 0$,
  \begin{align}\label{eq:hitting_X}
  	X_n = x \quad \iff \quad T_x \le n < T_{x+1}.
  \end{align}
  Define 
  $$
  T^{ \leftarrow }(n) = \inf \{  x \ge 0 \: : \: T_x \ge n  \}.
  $$
  By \eqref{eq:hitting_X},
  $$
  X_n + 1= T^{ \leftarrow }(n+1).
  $$
  
  \begin{prop}\label{prop:distr_tau} The sequence $(\tau_x)_{x \ge 0}$ is independent, and the distribution of 
  	$\tau_x$ is given by
  	\begin{align}\label{eq:distr_tau}
  		P_\omega( \tau_x = n ) = \omega_{n-1}^x - \omega_{n}^x \quad \forall n \ge 1.
  	\end{align}
  	Moreover, 
  	\begin{align}\label{eq:rel_x_t}
  	P_\omega(X_n = x) = \sum_{k= 1}^n P_\omega ( T_x = k) \omega_{n - k}^x.
  	\end{align}
  \end{prop}
  
  \begin{proof}
  	Fix integers $x\ge0$ and $n_0,\dots,n_x\ge1$, and set $m_j =n_0+\cdots+n_j-1$ for $j \ge 0$.
  	Note that
  	\begin{align*}
  		&\{  \tau_0=n_0,\dots,\tau_x=n_x \} = \{  
  		Z_0=(0,n_0-1),\ Z_{m_0}=(0,0),\ Z_{m_0+1}=(1,n_1-1),\ \ldots, \\
  		&Z_{m_{x-1}}=(x-1,0),\ Z_{m_{x-1}+1}=(x,n_x-1)
  		\},
  	\end{align*}
  	so using the Markov property we obtain
  	\begin{align*}
  		&P_\omega(\tau_0=n_0,\dots,\tau_x=n_x) \\
  		&\qquad= P_\omega(Z_0=(0,n_0-1))\,
  		\prod_{w=1}^{x}
  		P_\omega(Z_{m_{w-1}+1}=(w,n_{w}-1)\mid Z_{m_{w-1}}=(w-1,0)).
  	\end{align*}
  	As the transition probabilities are given by $\cK((w,n_{w}-1),(w-1,0))=\omega_{n_{w}-1}^{w}-\omega_{n_{w}}^{w}$, 
  	this implies both the independence of $(\tau_x)_{x\ge0}$ and \eqref{eq:distr_tau}.
  	
  	Finally, using \eqref{eq:hitting_X} and the independence of $T_x = \sum_{w=0}^{x-1} \tau_w$
  	and $\tau_x$
  	,
  	we have
  	\begin{align*}
  		P_\omega(X_n=x)
  		&=\sum_{k=0}^n P_\omega(T_x=k) P_\omega(\tau_x\ge n-k+1) \\
  		&=\sum_{k=0}^n P_\omega(T_x=k)\sum_{j\ge n-k+1}(\omega_{j-1}^{x}-\omega_{j}^{x})
  		=\sum_{k=0}^n P_\omega(T_x=k)\omega_{n-k}^{x},
  	\end{align*}
  	as wanted.
  \end{proof}
  
  In the sequel we assume that $A_x < \infty$ for $x \ge 0$.
  Using \eqref{eq:distr_tau} we obtain the formulas
  \begin{align}\label{eq:m_x}
  m_x = E_\omega( \tau_x ) = \sum_{n=1}^\infty n (   \omega_{n-1}^x - \omega_{n}^x  ) 
  = \sum_{n=0}^\infty \omega_{n}^x
  \end{align}
  and
  \begin{align}\label{eq:mu_x}
  	\mu_x = E_\omega(T_x) = \sum_{w=0}^{x-1}  \sum_{n=0}^\infty \omega_{n}^{w} = x + 
  	\sum_{w=0}^{x-1}  \sum_{n=1}^\infty \omega_{n}^{w}.
  \end{align}
  Set 
  $$
  \mu_{-} = \inf_{x \ge 1} \mu_x / x \quad \text{and} \quad 
  \mu_{+} = \sup_{x \ge 1} \mu_ x /x.
  $$
  Assuming (A1) 
  we have $0 < \mu_{-} \le \mu_{+} < \infty$.
  
  Using \eqref{eq:distr_tau} and the definition of $A_x$ in \eqref{eq:A_x} we see that
  for $n \ge 1$,
  \begin{align}\label{eq:tail_tau}
  	P_\omega( \tau_x \ge n ) = \omega_{n-1}^x \le C A_x n^{ -\beta_* }.
  \end{align}
  Recall from \eqref{eq:Mn} that
  $$
  M_n = (\mu_x)^{ \leftarrow }(n).
  $$
  From
  \eqref{eq:mu_x} it is clear that $M_n \le n$. Further, assuming (A1), it follows 
  from \cite[Lemma A.1]{LS11} that 
  \begin{align}\label{eq:Mn_lim}
  	\lim_{n\to\infty} n^{-1} M_n = \mu^{-1}.
  \end{align}
  
   If $\beta_* > 2$, then a computation using \eqref{eq:distr_tau} gives
  \begin{align}\label{eq:tau_x_moment_2}
  	E_\omega(\tau_x^2) = \sum_{n=1}^\infty n^2  (   \omega_{n-1}^x - \omega_{n}^x  ) 
  	= \sum_{n=0}^\infty (2n + 1) \omega_n^x.
  \end{align}
  Further,
  \begin{align}\label{eq:rel_sx_ax}
  	s_x^2 := \text{Var}_\omega( \tau_x ) = \sum_{n=0}^\infty (2n + 1) \omega_n^x - \biggl[
  	\sum_{n=0}^\infty 
  	\omega_n^x \biggr]^2
  	\le C A_x.	
  \end{align}
  Finally, set
  \begin{align}\label{eq:sigma}
  	\begin{split}
  		\sigma_x^2 &= \text{Var}_\omega(T_x) = \sum_{w=0}^{x-1}  s_x^2 = \sum_{w=0}^{x-1} \biggl\{ 
  		\sum_{n=0}^\infty (2n + 1) \omega_n^w - \biggl[
  		\sum_{n=0}^\infty 
  		\omega_n^w \biggr]^2
  		\biggr\}, \\
  		\sigma_{-}^2 &= \inf_{ x \ge 1 } \sigma_x^2 / x, \quad \sigma_{+}^2 = \sup_{ x \ge 1 } \sigma_x^2 / x.
  	\end{split}
  \end{align}
  For $\beta_* > 2$ we have $\sigma_x^2 < \infty$ and 
  assuming (A2) we have $0 < \sigma_{-}^2 \le \sigma_{+}^2 < \infty$.
  
\subsection{Proof of Theorem \ref{thm:slln}}

\begin{lem}\label{prop:slln_tau} Suppose that the assumptions of Theorem 
\ref{thm:slln} hold. Then,
$$
\lim_{ x \to \infty }  x^{-1} T_x = \mu \quad \text{a.s.}
$$
\end{lem}

\begin{proof} Start by decomposing
$$\tau_x = \hat{\tau}_x + R_x,$$
where 
$$
\hat{\tau}_x = \tau_x \mathbf{1}_{ \{ \tau_x \le x \} } \quad 
\text{and} \quad R_x = \tau_x \mathbf{1}_{ \{ \tau_x > x \} }
$$
Let 
\begin{align}\label{eq:cond_r}
	r >  \frac{1}{\beta_* - 1 - \lambda}.
\end{align}
Let $u, u' \in (0,1)$ be numbers whose values will be determined later.

Let $x \ge 1$. Using \eqref{eq:tail_tau} and \eqref{eq:cond_ax_slln}, we obtain
\begin{align}\label{eq:exp_Rx}
	\begin{split}
	E_\omega(R_x) &= \sum_{n = x + 1}^\infty P_\omega( \tau_x = n ) n
	= (x + 1) P_\omega (  \tau_x \ge x + 1 )  + 
	\sum_{n = x + 2}^\infty P_\omega( \tau_x \ge n ) \\
	&\le C (x + 1)^{  \lambda - \beta_* + 1 }.
		\end{split}
\end{align}
Therefore, 
\begin{align*}
	\sum_{w=0}^{x-1} E_\omega(R_w) \le C \sum_{w = 0}^{x-1} (w + 1)^{ \lambda -  \beta_* + 1  } 
	\le C \max \{ \log(x),  x^{  2 + \lambda - \beta_*  } \}.
\end{align*}
In particular,
\begin{align}\label{eq:exp_R_vanish}
\lim_{x \to \infty} x^{-1} \sum_{w=0}^{x-1} E_\omega(R_w) = 0.
\end{align}

Next, define 
$$
\bar{\tau}_x = \tau_x - E_\omega(\tau_x) \quad \text{and} \quad 
\tilde{\tau}_x = \hat{\tau}_x - E_\omega( \hat{\tau}_x ),
$$
and decompose 
\begin{align*}
	x^{-1}\sum_{w=0}^{x-1} \bar{\tau}_w = 
	x^{-1}  \sum_{w=0}^{x-1} \tilde{\tau}_w  + x^{-1} \sum_{w=0}^{x-1} R_w 
	+ x^{-1} \sum_{w=0}^{x-1} E_\omega( \hat{\tau}_w - \tau_w  ) =: E_1(x) + E_2(x) + E_3(x).
\end{align*}
We will show that each of these three terms converges to zero almost surely.

\noindent\emph{Convergence of $E_3(x)$:} Since $E_{3}(x) = -x^{-1} \sum_{w=0}^{x-1} E_\omega(R_w) $, we have already established in \eqref{eq:exp_R_vanish} that
$$
\lim_{x \to \infty} E_3(x) = 0. 
$$

\noindent\emph{Convergence of $E_2(x)$:} Set $x_k = \lceil k^r \rceil$. Using 
Markov's inequality and \eqref{eq:exp_Rx}, we obtain 
\begin{align*}
	P_\omega \biggl(  
	\sum_{w=0}^{x_k - 1} R_w \ge x_k^u \biggr)
	\le C x_k^{-u} \sum_{w=0}^{x_k - 1} 
	(w + 1)^{\lambda - \beta_* + 1} 
	\le C k^{-r u}  \max\{  \log(k + 1), k^{ r( 2 + \lambda -  \beta_*   ) }
	\}
\end{align*}
for any $k \ge 1$. It follows from \eqref{eq:cond_r} that this upper bound is summable 
whenever $u$ is sufficiently close to $1$. Fix such $u$. Then, by the Borel--Cantelli Lemma,
\begin{align}\label{eq:first_bc}
	\lim_{k \to \infty} x_k^{-1} \sum_{w=0}^{x_k - 1} R_w = 0 \quad \text{a.s.}
\end{align}

Next, for $x_k \le x < x_{k+1}$, we estimate
\begin{align}\label{eq:decomp_bc}
\sum_{w=0}^{x-1} R_w \le \sum_{w=0}^{x_k-1} R_w + \sum_{w=x_k}^{x_{k+1}-1} R_w.
\end{align}
Using \eqref{eq:exp_Rx} and $r > 1$,
\begin{align*}
	E_\omega \biggl[  \sum_{w=x_k}^{x_{k+1}-1} R_w \biggr]
	&\le C \sum_{w=x_k}^{x_{k+1}- 1} (w + 1)^{ \lambda  - \beta_* + 1}
	\le C ( x_{k+1} - x_k )  x_{k}^{ \lambda - \beta_*  + 1 }  
	\le C k^{r - 1 + r (  \lambda - \beta_* + 1  ) }.
\end{align*}
Consequently,
\begin{align*}
	P_\omega \biggl(  \sum_{w=x_k}^{x_{k+1}-1} R_w \ge x_k^{u'} \biggr) 
	\le C k^{ - u' r }  k^{r - 1 + r (  \lambda - \beta_* + 1  ) } 
	= C k^{  r(  2 - u' + \lambda - \beta_* ) - 1 }.
\end{align*}
We fix $u'$ sufficiently close to $1$ such that
\begin{align*}
	2 - u' + \lambda -  \beta_*  < 0,
\end{align*}
which is possible by our assumption that $\lambda < \beta_* - 1$.
Then it follows by the Borel--Cantelli lemma that
\begin{align}\label{eq:second_bc}
\lim_{k \to \infty} x_k^{-1} \sum_{w=x_k}^{x_{k+1}-1} R_w = 0 \quad \text{a.s.}
\end{align}
Combining \eqref{eq:first_bc}, \eqref{eq:decomp_bc}, and \eqref{eq:second_bc}, we have
$$
\lim_{x \to \infty} E_2(x) = 0 \quad \text{a.s.}
$$

\noindent\emph{Convergence of $E_1(x)$:} By \eqref{eq:tail_tau} and \eqref{eq:cond_ax_slln},
\begin{align*}
	E_\omega( \tilde{\tau}_x^2 ) &\le E_\omega(  \hat{\tau}_x^2  )
	\le C \sum_{n=1}^{x} n  P_\omega( \tau_x  \ge n  ) 
	= C A_x \sum_{n=1}^{x} n^{ 1 - \beta_*  } 
	\le C \max \{  \log (x)  , x^{  2 - \beta_* }  \} x^\lambda.
\end{align*}
Consequently,
$$
\sum_{ x = 1 }^\infty x^{-2} E_\omega( \tilde{\tau}_x^2 ) < \infty,
$$
since $\lambda < \min\{ 1, \beta_* - 1 \}$.
Now it follows from Kolmogorov's variance criterion that 
$$
\lim_{x \to \infty} E_1(x) = 0 \quad  \text{a.s.}
$$

We have shown that
$$
\lim_{x \to \infty} x^{-1}\sum_{w=0}^{x-1} \bar{\tau}_w = 0 \quad \text{a.s.}
$$
Therefore, $x^{-1} T_x \to \mu$ a.s. follows by (A1).

\end{proof}

\begin{proof}[Proof of Theorem \ref{thm:slln}] Recall that $X_n + 1 = T^{\leftarrow}(n+1)$.
Therefore, by \cite[Lemma A.1]{LS11}, 
\begin{align*}
	\lim_{x \to \infty} x^{-1} T_x = \mu \iff \lim_{n \to \infty} n^{-1} X_n = \mu^{-1}.
\end{align*}
We obtain Theorem \ref{thm:slln} by combining this relation with the result of
Proposition \ref{prop:slln_tau}.
\end{proof}

\subsection{Proof of Theorem \ref{thm:clt}} 
First, we establish a CLT for $(T_x)$ by verifying the Lindeberg condition, and then deduce 
from this a CLT for $(X_n)$.

\begin{lem}\label{lem:clt_T} Suppose that the 
	assumptions in Theorem \ref{thm:clt} hold. Then,
\begin{align}\label{eq:clt_Tx}
	\frac{ T_x - \mu_x }{ \sigma_x } \stackrel{\cD}{\to} N(0,1), \quad \text{as $x \to \infty$},
\end{align}
where $\stackrel{\cD}{\to}$ means convergence in distribution.
\end{lem}

\begin{proof} By the well-known Lindeberg condition for the CLT,
it suffices to prove that, for each $\ve > 0$,
\begin{align*}
	\lim_{x \to \infty} \sigma_x^{-2} \sum_{w=0}^{x-1} E_\omega[ ( \tau_w - m_w  )^2 \mathbf{1}_{ | \tau_w - m_w| \ge \ve \sigma_x }  ] = 0.
\end{align*}
Since $\sigma_x^2 \in [ \sigma_{-} x, \sigma_{+} x ]$ where $\sigma_{-}^2 > 0$ by (A2),
it suffices to prove
\begin{align}\label{eq:lindeberg-1}
	\lim_{x \to \infty} x^{-1} \sum_{w=0}^{x-1} E_\omega[ ( \tau_w - m_w  )^2 \mathbf{1}_{ | \tau_w - m_w| \ge \ve \sqrt{x} } ] = 0.
\end{align}
Further, since we are assuming that $A_x = o( \sqrt{x} )$,
\begin{align*}
	\sum_{w=0}^{x-1} E_\omega[ ( \tau_w - m_w  )^2 \mathbf{1}_{ | \tau_w - m_w| \ge \ve \sqrt{x} } ] = \sum_{w=0}^{x-1} E_\omega[ ( \tau_w - m_w  )^2 \mathbf{1}_{ \tau_w \ge m_w + \ve \sqrt{x} } ]
\end{align*}
holds whenever $x$ is sufficiently large.

Using summation by parts 
together with $m_w \le CA_w$,
\begin{align*}
	&E_\omega[ ( \tau_w - m_w  )^2 \mathbf{1}_{ \tau_w \ge m_w + \ve  \sqrt{x} } ] \\
	&\le ( \lceil m_w + \ve \sqrt{x} \rceil - m_w )^2 \omega_{  \lceil m_w + \ve \sqrt{x} \rceil - 1 }^w + \sum_{n=  \lceil m_w + \ve \sqrt{x} \rceil  }^\infty \omega_n^w [ (n+1 - m_w)^2 - (n - m_w)^2 ] \\
	&\le C\biggl( x A_w (  m_w + \ve \sqrt{x} )^{ - \beta_* } + A_w  m_w \sum_{n=  \lceil m_w + \ve \sqrt{x} \rceil }^\infty 
	n^{-\beta_* + 1}
	   \biggr)  \le C   A_w^2 x^{  1 - \beta_* / 2 }.
\end{align*}
This yields \eqref{eq:lindeberg-1} assuming \eqref{eq:ax_ord_clt}.
\end{proof}

\begin{proof}[Proof of Theorem \ref{thm:clt}]
 Write $c_n(x) = x  \tilde{\sigma}\sqrt{n} + n \mu^{-1}$.
Using the relation
$$
X_n \le x \iff T_{x + 1} > n,
$$
we find that
\begin{align*}
	P_\omega \biggl( \frac{X_n -  n \mu^{-1} }{ \tilde{\sigma} \sqrt{n} } \le x \biggr) &= 
	P_\omega (  X_n \le \round{ c_n(x) } ) 
	\\
	&= 
	1 - P_\omega \biggl(  
	\frac{ T_{  \round{ c_n(x)} + 1 } - \mu_{ \round{c_n(x)} + 1 } }{ \sigma_{ \round{c_n(x)} + 1 } } 
	\le \frac{n - \mu_{ \round{c_n(x)} + 1 } }{ \sigma_{ \round{ c_n(x) } + 1 } }
	\biggr).
\end{align*}
Denote by $\Phi(z)$ the distribution function of $N(0,1)$.
If $(Z_n)$ is a sequence of random variables with $Z_n \stackrel{\cD}{\to} N(0,1)$, and if $(z_n)$ is a sequence of real numbers with $z_n \to z$, we have that 
$P(Z_n \le z_n) \to \Phi(z)$ by the continuity of $\Phi(z)$.
Therefore, by Lemma \ref{lem:clt_T}, it suffices to prove
\begin{align}\label{eq:limit_scaled_norm}
	\frac{n - \mu_{ \round{c_n(x)} + 1 } }{ \sigma_{ \round{ c_n(x) } + 1 } } 
	\to - x, \quad n \to \infty.
\end{align}
Recall that $\tilde{\sigma}^2 = \sigma^2 / \mu^3$. By (A1) and (A2) we have 
$(\mu c_n(x))^{-1} \mu_{ \round{c_n(x)} + 1 } \to 1$ and 
$ \sigma^{-2} c_n(x)^{-1} \sigma_{ \round{ c_n(x) } + 1  }^2 \to 1$ as $n \to \infty$. Moreover,
as $n \to \infty$,
\begin{align*}
	\frac{n - \mu c_n(x)}{ \sigma ( c_n(x) )^{1/2} } 
	= - \frac{   x \tilde{\sigma} \sqrt{n} \mu  }{ \sigma \sqrt{ x  \tilde{\sigma} n^{1/2}  + n \mu^{-1} } } \to - \frac{x \tilde{\sigma} \mu  }{ \sigma \mu^{-1/2} } = -x.
\end{align*}
Combining these properties gives \eqref{eq:limit_scaled_norm}.
\end{proof}

\subsection{Proof of Theorem \ref{thm:llt}}\label{sec:llt_proof} 
Let $\beta_* > 2$. Define
\begin{align*}
	f_x(n) &= \phi_{ \mu_x, \sigma_x^2 }(n) = ( 2 \pi \sigma_x^{2}  )^{-1/2} e^{  - \frac{( n - \mu_x )^2}{ 2 \sigma_x^2 }   }, \\
	g_x(n) &= ( 2 \pi  n \sigma^2 / \mu  )^{-1/2} e^{   - \frac{(n - \mu_x)^2}{ 2 n \sigma^2 / \mu }    }, \\
	h_n(x) &= \phi_{ M_n, n \tilde{\sigma}^2 }(x) = ( 2 \pi n \tilde{\sigma}^2 )^{-1/2} e^{ - \frac{(x - M_n)^2}{2 n \tilde{\sigma}^2}  },
\end{align*}
where $\tilde{\sigma}^2 = \sigma^2 / \mu^3$. The proof of the LLT \eqref{eq:local_limit} is based on 
the strategy of \cite{LS11}. We begin by outlining the main steps.

\begin{itemize}[leftmargin=20pt]
	\item[(1)] \emph{Local limit theorem for $T_x$, Lemmas \ref{lem:llt_hitting} and  \ref{lem:estim_e1}.} Using the representation 
	$T_x = \sum_{w = 0}^{x-1} \tau_w$, we apply a Gaussian LLT from \cite{KLN23} in the case of independent variables with 
	finite $2 + \delta$ moments to obtain
	\begin{align*}
	P_\omega( T_x = n ) = f_x(n) + O(x^{ - \frac{1 + \delta}{2}  }), \quad \text{as $x \to \infty$},
	\end{align*}
	for any $\delta \in (0, \min \{ \beta_* - 2, 1 \} )$.
	From this we deduce that
	\begin{align*}
	\sup_{x \ge 1} x^\eta |  P_\omega( T_x = n ) - f_x(n)   | = O(n^{ \eta - \frac{1 + \delta}{2} }), 
	\quad \text{as $n \to \infty$}.
	\end{align*}
	
	\item[(2)] \emph{Approximating $f_x(n)$ by $g_x(n)$, Lemma \ref{lem:estim_e2}.} We show that 
	\begin{align*}
	\max_{1 \le x \le n} x^\eta | f_x(n)  - g_x(n)  | = o(n^{-1/2}), \quad \text{as $n \to \infty$}.
	\end{align*}
	\item[(3)] \emph{Approximating $g_x(n)$ by
	$\mu^{-1}h_n(x)$, Lemma \ref{lem:estim_e3}.} 
	Using \eqref{eq:L_ass}, 
	we show that 
	\begin{align*}
		\sup_{1 \le x \le n} x^\eta | g_x(n) - \mu^{-1} h_n(x) | = o(n^{-1/2}), \quad \text{as $n \to \infty$}.
	\end{align*}
	Moreover, using \eqref{eq:Mn_lim} we obtain
	$$
	\sup_{x > n} x^\eta h_n(x)  = o(n^{-1/2}), \quad \text{as $n \to \infty$}.
	$$
	\item[(4)] \emph{Inverting the LLT for $T_x$ to obtain an LLT for $X_n$.} Combining the estimates from steps (1)-(3) we derive
	\begin{align*}
		\sup_{x \ge 1} x^\eta \biggl| P_\omega(T_x = n) - \mu^{-1} h_n(x) \biggr| = O(n^{ \eta - \frac{1 + \delta}{2} }) 
		+ o(n^{-1/2}), \quad \text{as $n \to \infty$.}
	\end{align*}
	We then apply \eqref{eq:rel_x_t} to deduce the desired LLT for $(X_n)$.
\end{itemize}

We now proceed to implement the above steps rigorously.  
To this end, we decompose
 \begin{align}\label{eq:decomp}
 	\begin{split}
 	&P_\omega(T_x = n) - \mu^{-1} h_n(x) \\
 	&=  [P_\omega(T_x = n) - f_x(n)]
 	+ [ f_x(n) - g_x(n) ] + [g_x(n) - \mu^{-1} h_n(x) ] \\
 	&=: E_1(x,n) + E_2(x,n) + E_3(x,n).
 	\end{split}
 \end{align} 
 
 \subsubsection{Estimate on $E_1(x,n)$} 
 Recall the definitions of $A_x$ and $A_x'$ in \eqref{eq:A_x}.
 The following lemma establishes a local limit theorem 
 for $T_x$.

 \begin{lem}\label{lem:llt_hitting} Suppose that (A2) holds, and that
	\begin{align}\label{eq:aperiod}
		K_x = \sup_{n \ge 2} \frac{  \omega_{n}^x - \omega_{n+1}^x }{ \omega_1^x - \omega_2^x }
		+ \frac{ \omega_1^x - \omega_2^x  }{ 1 - \omega_1^x }
	\end{align}
	satisfies
	$$
	\frac{1}{x} \sum_{w = 0}^{x-1} K_w = O(1).
	$$
	Moreover, assume that
	\begin{align}\label{eq:ass_B}
		\frac1x \sum_{w=0}^{x-1} B_w^{3 + \ve} = O(1)
	\end{align}
	for some $\ve > 0$, where 
	\begin{align*}
		B_w = \begin{cases}
			A_w', &\text{if $\beta_* \in (2,3]$,} \\
			A_w, &\text{if $\beta_* > 3$.}
		\end{cases}
	\end{align*}
	Then, for any $\delta \in (0, \min \{ \beta_* - 2, 1 \} )$,
	$$
	\sup_{n \ge 0} |E_1(x,n)|  = O( x^{ - \frac{1 + \delta }{2} }  ), \quad \text{as $x \to \infty$.} 
	$$
\end{lem}

\begin{proof} Recall from Proposition \ref{prop:distr_tau} that $(\tau_x)_{x \ge 0}$ is an independent 
	sequence, and that
	$$
	T_x = \sum_{w=0}^{x-1} \tau_w, \quad \mu_x = E_\omega(T_x) = \sum_{w = 0}^{x-1} 
	m_w, \quad \sigma_x^2 = \text{Var}_\omega(T_x),
	$$
	where $m_w = E_\omega(\tau_w)$.
	Let $\delta \in (0, \min \{ \beta_* - 2, 1 \} )$. In the rest of this proof, we denote by $C_\delta$ 
	a generic positive constant depending on $\delta$, independent of $n$ and $x$.
	Since $\beta_* > 2$, it follows from \eqref{eq:tail_tau} that
	\begin{align}\label{eq:moment_tau}
		E_\omega(\tau_x^{2 + \delta}) \le 
		C_\delta \sum_{n=1}^\infty n^{1 + \delta} P_\omega(  \tau_x \ge n ) \le 
		C_\delta A_x \sum_{n=1}^\infty n^{ 1 + \delta  - \beta_* } < \infty.
	\end{align}
	Moreover, from (A2) it follows that there exists $x_0 \ge 1$ such that 
	$\sigma_x^2 > 1$ whenever $x \ge x_0$.
	Therefore, an application of \cite[Theorem 1.1]{KLN23} (with $a = 0$ and $d = 1$) yields the following upper bound for any 
	$x \ge x_0$:
	\begin{align}\label{eq:bound_E1-1}
		\begin{split}
			\sup_{n \ge 0} | E_1(x,n) | &\le  
			C_\delta \biggl(  
			\sum_{w=0}^{x-1}  E_\omega[ \tau_w ^{2 + \delta} ]
			\biggr)^{ - \frac{1 + \delta}{2}  } 
			+ C_\delta \sigma_x^{-4} \biggl( \sum_{w=0}^{x-1} E[  \tau_w^{2 + \delta} ]  \biggr)^{ \frac{3 - \delta}{2} } \\
			&+ C_\delta \sigma_x^{-2} \gamma_x^{-1} e^{ - \frac{ \sigma_x^2 \gamma_x^2 }{2}  } 
			+ C_\delta \gamma^{-1}_x \iota_x^{-1} e^{ - \frac{\gamma^2_x \iota_x }{\pi}  } = : 
			I + II + III + IV,
		\end{split}
	\end{align}
	where
	$$
	\gamma_x = 3^{ - 1/\delta } \biggl(  \sum_{w = 0}^{x-1} E[\tau_w^{2 + \delta}]   \biggr)^{ - \frac{1}{2 + \delta} },
	$$
	and 
	$$
	\iota_x = \sum_{w=0}^{x-1} 2 \sum_{n \ge 1} P_\omega( \tau_w = n ) P_\omega( \tau_w = n + 1 ).
	$$
	
	It remains to estimate the four terms in \eqref{eq:bound_E1-1}.
	Using (A2) and $\tau_w \ge 1$ we obtain the following for any $x \ge 1$:
	\begin{align*}
		0< x \sigma_{-}^2 \le 
		\sigma_x^2 = \text{Var}_\omega(T_x) \le \sum_{w=0}^{x-1}  E_\omega[ \tau_w^2 ] \le \sum_{w=0}^{x-1}  E_\omega[  \tau_w^{2 + \delta} ].
	\end{align*}
	Consequently, 
	\begin{align}
		|I| \le C x^{ - \frac{1 + \delta }{2} }.
	\end{align}
	
	Recalling \eqref{eq:rel_ax_ax'}, it follows from \eqref{eq:ass_B} that 
	\begin{align}\label{eq:avg_A_ord}
		x^{-1} \sum_{w=0}^{x-1} A_w = O(1).
	\end{align}
	Combined with \eqref{eq:moment_tau} and (A2), this gives
	$$
	|II| \le C x^{-2} \biggl(  \sum_{w=0}^{x-1} A_w  \biggr)^{ \frac{3 - \delta }{2} } 
	= C x^{  -\frac{1 + \delta}{2}  } 
	\biggl( x^{-1}  \sum_{w=0}^{x-1} A_w  \biggr)^{ \frac{3 - \delta }{2} } \le C_\delta x^{ - \frac{1 + \delta }{2} }.
	$$
	Using \eqref{eq:moment_tau} and \eqref{eq:avg_A_ord} we also find that
	\begin{align}\label{eq:gamma_x_lb}
		\gamma_x \ge C_\delta' x^{ - \frac{1}{2 + \delta} },
	\end{align}
	for some constant $C_\delta' > 0$ depending on $\delta > 0$. So,
	\begin{align}\label{eq:estim_iii}
		|III| \le C_\delta x^{  - 1  + \frac{1}{2 + \delta}  } \exp \biggl(  - C'_\delta   x^{  1  -  \frac{2}{2 + \delta}  }  \biggr).
	\end{align}
	for some constants $C_\delta, C_\delta' > 0$. In particular,
	$$
	|III| = O(  x^{ - \frac{1 + \delta }{2} }  ), \quad \text{as $x \to \infty$}.
	$$
	
	It remains to estimate $IV$. For this we use \eqref{eq:aperiod} and (A2). For brevity, denote
	$$
	S_x 
	= \sum_{w=0}^{x-1} P_\omega( \tau_w = 1 )   P_\omega( \tau_w = 2 )  
	= \sum_{w=0}^{x-1} (1 - \omega_1^w) (\omega_{1}^w - \omega_2^w).
	$$
	We decompose
	\begin{align*}
		\text{Var}_\omega(\tau_w) &= \sum_{n=1}^\infty (n -  m_w )^2 P_\omega(\tau_w = n) = (1 - m_w)^2 (1 - \omega_1^w)
		+  \sum_{n=2}^\infty (n - m_w)^2 ( \omega_{n - 1}^w - \omega_{n}^w ) \\
		&=: R_1(w) + R_2(w).
	\end{align*}
	For $r \in (0, 1/8)$,
	\begin{align*}
		R_1(w) &= \biggl\{  1  - \sum_{n=1}^\infty n( \omega_{n-1}^w - \omega_{n}^w )  \biggr\}^2 (1 - \omega_1^w)
		\le  \biggl\{  \sum_{n=2}^\infty n( \omega_{n-1}^w - \omega_{n}^w )  \biggr\}^2 (1 - \omega_1^w) \\
		&\le C_r (K_w + 1)^{2r}  [ (1 - \omega_1^w) ( \omega_1^w - \omega_2^w ) ]^{2r} \biggl\{  \sum_{n=2}^\infty n( \omega_{n-1}^w - \omega_{n}^w )^{1-r}  \biggr\}^2,
	\end{align*}
	where \eqref{eq:aperiod} was used in the last inequality and 
	$C_r > 0$ is a constant depending on $r$ but independent of 
	$x$ and $n$.
	Applying (generalized) H\"{o}lder's inequality 
	and $\frac{1}{x} \sum_{w = 0}^{x-1} K_w = O(1)$
	we then obtain
	\begin{align*}
		\frac1x \sum_{w=0}^{x-1} R_1(w) &\le C_r (x^{-1} S_x  )^{2r} 
		\biggl[ 1 + x^{-1}\sum_{w=0}^{x-1} K_w \biggr]^{2r}
		\biggr[ 
		x^{-1}
		\sum_{w=0}^{x-1} 
		\biggl\{  \sum_{n=2}^\infty n( \omega_{n-1}^w - \omega_{n}^w )^{1-r}  \biggr\}^{ \frac{2}{1-4r}  } 
		\biggr]^{  1 - 4r } \\
		&\le C_r (x^{-1} S_x  )^{2r} 
		\biggr[ 
		x^{-1}
		\sum_{w=0}^{x-1} 
		\biggl\{  \sum_{n=2}^\infty n( \omega_{n-1}^w - \omega_{n}^w )^{1-r}  \biggr\}^{ \frac{2}{1-4r}  } 
		\biggr]^{  1 - 4r }.
	\end{align*}
	
	Further, we use \eqref{eq:aperiod} to derive
	\begin{align*}
		R_2(w) &\le \sum_{n=2}^\infty n^2 ( \omega_{n-1}^w - \omega_{n}^w ) 
		+ m_x^2 \sum_{n=2}^\infty ( \omega_{n-1}^w - \omega_{n}^w) \\
		&\le C_r (K_w + 1)^{6r} [ (1 - \omega_1^w) ( \omega_1^w - \omega_2^w ) ]^{2r} \\
		&\times \biggl[
		\sum_{n=2}^\infty n^2 ( \omega_{n-1}^w - \omega_{n}^w )^{1-4r} 
		+ m_w^2 \sum_{n=2}^\infty ( \omega_{n-1}^w - \omega_{n}^w )^{1-4r}
		\biggr],
	\end{align*}
	and once more apply H\"{o}lder's inequality together with $\frac{1}{x} \sum_{w = 0}^{x-1} K_w = O(1)$:
	\begin{align*}
		\frac1x \sum_{w=0}^{x-1} R_2(w)
		&\le C_r (x^{-1} S_x)^{2r} 
		\biggl\{ x^{-1} \sum_{w=0}^{x-1} 
		\biggl[
		\sum_{n=2}^\infty n^2 ( \omega_{n-1}^w - \omega_{n}^w )^{1-4r} \\
		&+ m_w^2 \sum_{n=2}^\infty ( \omega_{n-1}^w - \omega_{n}^w )^{1-4r}
		\biggr]^{ \frac{1}{1-8r} }
		\biggr\}^{1 - 8r}.
	\end{align*}
	Gathering,
	\begin{align*}
		\frac{1}{x} \sum_{w=0}^{x-1} \text{Var}_\omega(\tau_w) \le 
		C_r ( x^{-1} S_x )^{2r} \biggl\{  
		E_1^{ 1 - 4r } + E_2^{1-8r}
		\biggr\},
	\end{align*}
	where 
	\begin{align*}
		E_1 &= x^{-1}
		\sum_{w=0}^{x-1} 
		\biggl[  \sum_{n=2}^\infty n( \omega_{n-1}^w - \omega_{n}^w )^{1-r}  \biggr]^{ \frac{2}{1-4r}  }, \\
		E_2 &= x^{-1} \sum_{w=0}^{x-1} 
		\biggl[
		\sum_{n=2}^\infty n^2 ( \omega_{n-1}^w - \omega_{n}^w )^{1-4r} + m_w^2 \sum_{n=2}^\infty ( \omega_{n-1}^w - \omega_{n}^w )^{1-4r}
		\biggr]^{ \frac{1}{1-8r} }.
	\end{align*}
	
	Now, if $\beta_* \in (2,3]$, then using $m_w \le C A_w'$ we obtain
	\begin{align}\label{eq:estim_e2}
		\begin{split}
			E_2
			&\le x^{-1} \sum_{w=0}^{x-1} 
			\biggl[
			C (A_w')^{  1 - 4r }
			\sum_{n=2}^\infty n^{2 - (1 + \beta_*)(1 - 4r)}
			+ C ( A_w' )^{3 - 4r} \sum_{n=2}^\infty 
			n^{- (1 + \beta_*)(1 - 4r)}
			\biggr]^{ \frac{1}{1-8r} } \\
			&\le C_r x^{-1} \sum_{w=0}^{x-1} (A_w')^{ \frac{3- 4r}{ 1 - 8r} },
		\end{split}
	\end{align}
	for sufficiently small $r$ depending only on $\beta_*$.
	A similar computation gives $$E_1 \le C_r x^{-1} \sum_{w=0}^{x-1} (A_w')^{ \frac{1 - r}{ 1 - 4r} }.$$
	In particular, it follows from \eqref{eq:ass_B} that 
	$E_1^{1-4r} + E_2^{1-8r} = O(1)$ whenever $r$ is sufficiently small. Fixing such $r$, we have 
	by (A2)
	that
	$S_x \ge c_0 x$ for some constant $c_0 > 0$, so that 
	$\iota_x \ge 2 S_x \ge 2 c_0 x$. Combining this inequality with
	\eqref{eq:gamma_x_lb}, we conclude that
	\begin{align*}
		|IV| \le C_\delta x^{ -1 + \frac{1}{2 + \delta}  } \exp \biggl( - C_\delta' x^{ 1 - \frac{2}{2 + \delta}  }  \biggr).
	\end{align*}
	Hence,
	$$
	|IV| = O (x^{  - \frac{1 + \delta}{2}  }), \quad \text{as $x \to \infty$.}
	$$
	
	Finally, if $\beta_* > 3$, then in place of \eqref{eq:estim_e2} we simply use
	\begin{align*}
		E_2 &\le x^{-1} \sum_{w=0}^{x-1} 
		\biggl[
		\sum_{n=2}^\infty n^2 ( \omega_{n-1}^w )^{1-4r} + m_w^2 \sum_{n=2}^\infty  (\omega_{n-1}^w)^{1-4r}
		\biggr]^{ \frac{1}{1-8r} } \\
		&\le C_r x^{-1} \sum_{w=0}^{x-1} (A_w)^{ \frac{3- 4r}{ 1 - 8r} } = O(1),
	\end{align*}
	whenever $r$ is sufficiently small.
	
\end{proof}

\begin{lem}\label{lem:estim_e1} Suppose that the environment satisfies (A1) in
	addition to
	the assumptions in Lemma \ref{lem:llt_hitting}. 
	Then,  for any $\delta \in (0, \min \{ \beta_* - 2, 1 \} )$ and any $\eta \in [0, 1/2)$, 
	\begin{align}\label{eq:e1_estim}
		\sup_{x \ge 1} x^\eta | E_1(x,n) | = O(n^{\eta - \frac{1 + \delta}{2} }), \quad \text{as $n \to \infty$.}
	\end{align}
\end{lem}

\begin{proof}
For any $\ve \in (0, \mu^{-1})$, it follows by Lemma \ref{lem:llt_hitting} that
\begin{align}\label{eq:ee1}
	\sup_{ x \ge  \round{ \ve n } } x^\eta | P_\omega(T_x = n) - f_x(n) | = O(n^{\eta - \frac{1 + \delta}{2} }),
\end{align}
as $n \to \infty$.

Then suppose that $x \le \round{\ve n}$. By (A1), there exist a constant $c_0 \in (0,1)$ and an integer $n_0 \ge 1$ such that 
$\mu_{ \round{ \ve n} } \le c_0 n$ whenever $n \ge n_0$. For such $n$, we have
\begin{align}\label{eq:estim_1}
	\frac{(n - \mu_x)^2}{\sigma_x^2} \ge \frac{( n - \mu_{ \round{\ve n} } )^2}{ \sigma_+^2 \ve n } \ge \frac{(1 - c_0)^2}{\sigma_+^2 \ve } n
\end{align}
with $\sigma_+^2 \in (0, \infty)$, where (A2) was used in the first inequality.
Therefore, by Markov's inequality,
\begin{align}\label{eq:ee2}
\begin{split}
x^\eta P_\omega(T_x = n) &\le x^\eta P_\omega( |T_x - \mu_x| \ge |n - \mu_x| )
\le  x^\eta (n - \mu_x)^{-2} \sigma_x^{2} \\
&\le C n^{\eta - 1} = O(n^{\eta - \frac{1 + \delta}{2} }).
\end{split}
\end{align}
Moreover, using \eqref{eq:estim_1} and (A2) we see that
\begin{align}\label{eq:ee3}
\begin{split}
	x^\eta f_x(n) &=   x^\eta  ( 2 \pi \sigma_x^2 )^{-1/2} \exp\biggl( - \frac{(n - \mu_x)^2}{2\sigma_x^2} \biggr) \le  
	C x^{ \eta - \frac12 } \exp( - C' n ) \\
	&\le C  \exp( - C' n ) = O(n^{\eta - \frac{1 + \delta}{2} }).
\end{split}
\end{align}
for some constants $C, C' > 0$. Combining \eqref{eq:ee1}, \eqref{eq:ee2}, and \eqref{eq:ee3}
gives \eqref{eq:e1_estim}.
\end{proof}

\subsubsection{Estimate on $E_2(x,n)$}

\begin{lem}\label{lem:estim_e2} Assume that (A1) and (A2) hold with 
$\theta_i(x)$ as in \eqref{eq:theta_ass}:
\begin{align*}
\theta_i(x) = o( x^{-\eta}  ( \log x )^{-1/2} ), \quad i = 1,2,
\end{align*}
where $\eta \in [0, 1/2)$. Then,
\begin{align}\label{eq:e2_estim}
	\max_{1 \le x \le n} x^\eta | E_2(x, n) | = o(n^{-1/2}).
\end{align}
	
\end{lem}

\begin{proof} We begin by decomposing
\begin{align*}
[1,n] \cap \bZ_+ &= \{  1 \le x \le n \: : \: |n - \mu_x| 
\le u ( n \log n )^{1/2}  \} \\
&\cup \{  1 \le x \le n \: : \: |n - \mu_x|  >  u ( n \log n )^{1/2}  \} =: I_{1,n} \cup I_{2,n},
\end{align*}
where $u > 1$ is a large constant whose value will be specified later.

First suppose that $x \in I_{1,n}$. Using the triangle inequality and 
$|e^{-x^2} - e^{-y^2}| \le |x - y|$, we have that
\begin{align*}
	n^{1/2} | E_2(x,n) | &= n^{1/2} | f_x(n) - g_x(n) | \\
&\le C | ( n / \sigma_x^2 )^{1/2}  - (\mu / \sigma^2)^{1/2}   | 
+ C 
\biggl| 
\exp\biggl( - \frac{(n - \mu_x)^2}{2\sigma_x^2}  \biggr)
- 	\exp\biggl( - \frac{(n - \mu_x)^2}{2 n \sigma^2 / \mu}  \biggr)
\biggr| \\
&\le C | ( n / \sigma_x^2 )^{1/2}  - (\mu / \sigma^2)^{1/2}   |  
+ C \biggl|  (n / \sigma_x^2)^{1/2}  \frac{n - \mu_x}{ \sqrt{n} } - 
(\mu / \sigma^2)^{1/2}
\frac{n - \mu_x}{
	\sqrt{n} 
}  \biggr| \\
&\le C \biggl( 1 + \frac{|n - \mu_x|}{ \sqrt{n} } \biggr) 
|  ( n / \sigma_x^2 )^{1/2}  - (\mu / \sigma^2)^{1/2}  |.
\end{align*}
Thus, for any $n$ with $u ( \log n )^{1/2} \ge 1$,
\begin{align*}
	n^{1/2} | E_2(x,n) | 
	 &\le C u ( \log n )^{1/2}  |  ( n / \sigma_x^2 )^{1/2}  - (\mu / \sigma^2)^{1/2}  | \\
	 &\le C  u ( \log n )^{1/2}  (\sigma^2 / \mu)^{1/2} | n / \sigma_x^2 - \mu / \sigma^2 | 
	 \le C u ( \log n )^{1/2}   \biggl| 
	 \frac{n - \mu_x}{\sigma_x^2} + d(x)
	 \biggr| \\
	 &\le C u ( \log n )^{1/2}  ( x^{-1} (n \log (n) )^{1/2} + | d(x) | ),
\end{align*}
where (A2) was used in the last inequality, and 
$$
d(x) = \frac{\mu_x}{\sigma_x^2} - \frac{\mu}{\sigma^2} =  \frac{\mu_x / x}{\sigma_x^2 / x} - \frac{\mu}{\sigma^2}.
$$
Since (A1) and (A2) are satisfied with
$\theta_i(x) = o( x^{-\eta}  ( \log x )^{-1/2} )$, $i = 1,2$,
 it follows that $d(x) = o(x^{-\eta} (\log x)^{-1/2})$. Further,
 for sufficiently large $n$ we have $n - u ( n \log n )^{1/2} \ge n / 2$, which yields
\begin{align}\label{eq:aux_x_n}
	x = ( \mu_x / x )^{-1} ( n - (n - \mu_x) ) \ge (2 \mu_+)^{-1} n.
\end{align}
Consequently, as $n \to \infty$,
$$
n^{1/2} \max_{x \in I_{1,n}} x^\eta  |E_2(x,n)| \le C ( \log n )^{1/2}  (  n^{\eta-1/2} ( \log n )^{ 1 / 2}   + \max_{1 \le x \le n}  x^\eta |d(x) |  ) = o(1).
$$

On the other hand, using (A2) and the definition of $I_{2,n}$,
\begin{align*}
	&n^{1/2} \max_{x \in I_{2,n}}  x^\eta |E_2(x,n)| \le n^{1/2} \max_{x \in I_{2,n}} x^\eta  |f_x(n)| + n^{1/2} \max_{x \in I_{2,n}}  x^\eta |g_x(n)|
	\\
	&\le C n^{1/2}  \max_{x \in I_{2,n}} x^{\eta - 1/2}  \exp \biggl( 
	 - \frac{(n - \mu_x)^2}{2\sigma_x^2} 
	\biggr)
	+ C  \max_{x \in I_{2,n}} x^{\eta}  \exp \biggl( 
	- \frac{(n - \mu_x)^2}{2 n \sigma^2 / \mu } 
	\biggr) \\
	&\le C n^{\eta - u^2 C' } 
\end{align*}
where $C' > 0$ is a constant independent of $u$. The desired estimate 
follows by choosing $u > (\eta / C' )^{  1/2 }$.
\end{proof}

\subsubsection{Estimate on $E_3(x,n)$}

\begin{lem}\label{lem:estim_e3} Assume (A1) 
	with $\theta_1(x) = o(1)$
	and \eqref{eq:L_ass} with $\eta \in [0, 1/2)$. Then, as $n \to \infty$,
	\begin{align}\label{eq:e3_estim}
		\max_{1 \le x \le n} x^\eta | E_3(x, n) | = o(n^{-1/2}),
	\end{align}
and, for some constant $C > 0$,
\begin{align}\label{eq:h_tails}
	\sup_{x > n} x^\eta h_n(x) = O( e^{-Cn} ).
\end{align}
\end{lem}

\begin{proof} Decompose
\begin{align*}
[1,n] \cap \bZ_+ &= \{  1 \le x \le n \: : \: |x - M_n| \le u b(M_n)  \} \cup 
 \{  1 \le x \le n \: : \: |x - M_n| > u b(M_n)  \} \\
 &=: J_{1,n} \cup J_{2,n},
\end{align*}
where $u \ge 2$ is a large constant to be specified later and $b(x) = (x \log x)^{1/2}$ is defined as in \eqref{eq:L_def}. Recall from \eqref{eq:Mn_lim} that $n^{-1} M_n \to \mu^{-1}$ as $n \to \infty$. 
Therefore, there exists $n_0$ such that for all $n \ge n_0$, 
$$
u b(M_n) \ge 2 \quad \text{and} \quad M_n \ge 2^{-1} \mu^{-1} n.
$$

For $x \in J_{1,n}$, using $|e^{-x^2} - e^{-y^2}| \le |x - y|$ we obtain
\begin{align*}
|E_3(x,n)| &= |g_x(n) - \mu^{-1} h_n(x) | \le C n^{-1/2} |(n - \mu_x) - \mu (M_n - x) |.
\end{align*}
Next, write 
\begin{align*}
(n - \mu_x) - \mu (M_n - x)  &= ( \mu_{M_n} - \mu_x ) - \mu(M_n - x) + (n - \mu_{M_n})  
\\
&= \sum_{w=x}^{M_n - 1} (m_w - \mu) + (n - \mu_{M_n}),
\end{align*}
with the convention that $\sum_{w=x}^{M_n - 1} (m_w - \mu) = - \sum_{w=M_n}^{x - 1} (m_w - \mu)$ 
if $x \ge M_n$. Note that, by the definition of $M_n$,
$|  n - \mu_{M_n} | = \mu_{M_n} - n \le m_{M_n - 1}$. Consequently, for $n \ge n_0$, we have 
\begin{align*}
&|(n - \mu_x) - \mu (M_n - x) |  \\
&\le \biggl| \sum_{w=x}^{M_n - 1} ( m_w - \mu )   \biggr| + m_{ M_n - 1 } 
\le 2 \max_{x' : |x' - M_n| \le u b(M_n)} \biggl| \sum_{w=x'}^{M_n - 1} ( m_w - \mu )   \biggr|  + \mu \\
&\le 2 \max_{\ell : |\ell | \le u b(M_n)} \biggl| \sum_{  w \in [ M_n ]_{  \ell  } } ( m_w - \mu )   \biggr|  + \mu = 
2 L(M_n; u) + \mu.
\end{align*}
where $[x]_\ell$ is defined as in \eqref{eq:L_def}.
Now, using \eqref{eq:L_ass} and \eqref{eq:Mn_lim} we obtain
$$
\max_{x \in J_{1,n}}  |(n - \mu_x) - \mu (M_n - x) | = O( L( M_n; u ) + \mu ) = o(n^{ 1/2 - \eta }), \quad 
\text{as $n \to \infty$.}
$$
Consequently, as $n \to \infty$,
\begin{align*}
n^{1/2}	\max_{x \in J_{1,n}} x^\eta |E_3(x,n)| = o(1).
\end{align*}	

Next, we consider $x \in J_{2,n}$. For $n \ge n_0$ we have $ 2^{-1} |x - M_n| \ge 1$.
Since $\inf_{x} ( \mu_{x+1} - \mu_x ) \ge 1$
and $M_n = (\mu_x)^{\leftarrow}(n)$, it 
follows from \cite[equation (A.2)]{LS11} that, whenever $n \ge n_0$,
$$
|\mu_x - n| \ge |x - M_n| - 1 \ge 2^{-1}|x - M_n|.
$$
Further, for $n \ge n_0$, we use $M_n \ge 2^{-1} \mu^{-1} n$ and $|x - M_n| > u b(M_n)$
to obtain
\begin{align*}
	n^{1/2}|E_3(x,n)| &\le C  \biggl( \exp \biggl(  - \frac{(n - \mu_x)^2}{ 2 n \sigma^2 / \mu } \biggr) 
	+ \exp \biggl(  - \frac{(x - M_n)^2}{ 2 n \tilde{\sigma}^2 } \biggr)  
	  \biggr) \\
	  &\le C \exp \biggl(  -  C' u^2 \frac{ M_n \log M_n  }{ n  } \biggr) 	  
	  \le C M_n^{- C' u^2 \frac12 \mu^{-1} } \le C n^{ - C' u^2 \frac12 \mu^{-1}  },
\end{align*}
for some constants $C, C ' > 0$ independent of $u$. Hence, for $u > 2 + \sqrt{ 2 \eta \mu / C' }$,
\begin{align*}
	n^{1/2}	\max_{x \in J_{2,n}} x^\eta |E_3(x,n)| = O( n^{  \eta - C' u^2 \frac12 \mu^{-1}   } ) = o(1), \quad 
	\text{as $n \to \infty$.}
\end{align*}	
This completes the proof of \eqref{eq:e3_estim}.

It remains to show \eqref{eq:h_tails}. Since
$M_n / n \to \mu^{-1} < 1$, there exist $n_1$ and $c \in (0,1)$ such that, 
whenever $x > n \ge n_1$, we have
$$
x - M_n \ge (1 - c)x.
$$
For such $n$ and $x$,
$$
x^\eta h_n(x) \le C x^\eta  n^{-1/2} 
\exp\biggl( - \frac{(x - M_n)^2}{2 n  \tilde{\sigma}^2 }  \biggr) 
\le C x^\eta  n^{-1/2}  e^{ - C' x }
$$
holds for some constants $C, C' >0$, which implies \eqref{eq:h_tails}.
\end{proof}

\subsubsection{Completing the proof of Theorem \ref{thm:llt}}
Assume that (B1)-(B4) in Theorem \ref{thm:llt} hold with
$$
\beta_* > 2 \quad \text{and} \quad 
0 \le \eta < \min \biggl\{  \frac12, \frac{\beta_* - 2}{2}   \biggr\}.
$$
Then there exists $\delta \in (0, \min \{ \beta_* - 2, 1 \} )$ such that $\eta - (1 + \delta) / 2 < -1/2$. 
Fix such $\delta$. Combining the decomposition \eqref{eq:decomp} with the estimates 
\eqref{eq:e1_estim}, \eqref{eq:e2_estim}, and \eqref{eq:e3_estim}, we obtain
\begin{align*}
	\max_{1 \le x \le n} x^\eta \biggl| P_\omega(T_x = n) - \mu^{-1} h_n(x) \biggr| = 
	O( n^{ \eta - \frac{1+\delta}{2} } )
	+
	o(n^{-1/2}) = o(n^{-1/2}), \quad \text{as $n \to \infty$.}
\end{align*}
Since $P_\omega(T_x = n) = 0$ for $x > n$, we conclude from \eqref{eq:h_tails} that
\begin{align*}
	\sup_{x \ge 1} x^\eta \biggl| P_\omega(T_x = n) - \mu^{-1} h_n(x) \biggr| =
	o(n^{-1/2}), \quad \text{as $n \to \infty$.}
\end{align*}
Combining this with \eqref{eq:rel_x_t} gives
\begin{align*}
	&\biggl|  P_\omega(X_n = x) - \mu^{-1} \sum_{\ell = 1}^n h_\ell(x) \omega_{n - \ell}^x  \biggr| 
	\le \sum_{\ell=1}^n \biggl| 
	P_\omega( T_x = \ell ) - \mu^{-1} h_\ell(x) 
	\biggr| \omega_{n - \ell}^x 
	\\
	&\le C  \hat{\theta}(n) 
	+ C x^{-\eta} A_x \sum_{\ell = 1}^{n-1} (n - \ell)^{ - \beta_* }
	\hat{\theta}(\ell),
\end{align*}
where $\hat{\theta}(n) = o(n^{-1/2})$. Recall that  $x^{-\eta} A_x = O(1)$ by \eqref{eq:ass_A}.
Moreover, as $n \to \infty$,
\begin{align*}
	 \sum_{\ell = 1}^{n-1} (n - \ell)^{ - \beta_* }
	\hat{\theta}(\ell) = o(n^{-1/2}).
\end{align*}
Therefore, as $n \to \infty$,
\begin{align*}
\biggl|  P_\omega(X_n = x) - \mu^{-1} \sum_{\ell = 1}^n h_\ell(x) \omega_{n - \ell}^x  \biggr| 
= o(n^{-1/2}),
\end{align*}
which proves \eqref{eq:local_limit}.

\subsection{Proofs of Theorems \ref{thm:clt_random} and \ref{thm:llt_random}} 
Recall that $(\omega_n^x)$ is 
defined by \eqref{eq:random_env}. For fixed $\omega \in \Omega$, we 
write $\beta(x) = \bar{\beta}(\Theta^x \omega)$, so that $A_x,A_x'$ in Theorem 
\ref{thm:llt} are defined by
\begin{align}\label{eq:ax_random}
	A_x &= \bar{A}(\Theta^x \omega) \quad \text{and} \quad 
	A_x' = \bar{A}'(\Theta^x \omega).
\end{align}
Recall also the definitions of $m_x$ and $s_x$ from \eqref{eq:m_x} and \eqref{eq:rel_sx_ax}.

  \begin{proof}[Proof of Theorem \ref{thm:clt_random}] \textbf{(i)} Since the random sequence $(A_x)$ is  stationary and by \eqref{eq:A_moment_slln}
  satisfies $\bE[  A_0^q ] < \infty$,
  	it follows by applying the Borel--Cantelli lemma (see
  	\cite[Lemma 4.1]{LS11}) that $A_x = o(  x^\lambda )$ for $\bP$-a.e. $\omega \in \Omega$
  	with $\lambda = q^{-1}$. Moreover, Birkhoff's ergodic implies that, for $\bP$-a.e. $\omega \in \Omega$,
  		\begin{align*}
  		\frac{1}{x} \sum_{w=0}^{x-1} m_w \to \bE[m_0], \quad \text{as $x \to \infty$,}
  	\end{align*}
  	where $\bE[m_0] > 1$. Hence the assumptions of Theorem \ref{thm:slln} are satisfied for 
  	$\bP$-a.e.\ $\omega \in \Omega$, and Theorem \ref{thm:clt_random}-(i) follows.
  	
  	\noindent\textbf{(ii)} Similar to the proof of (i), as $\bE[  A_0^2 ] < \infty$, we have 
  	 that $A_x = o(  \sqrt{x} )$ for $\bP$-a.e. $\omega \in \Omega$.
  	 Applying Birkhoff’s ergodic theorem yields, for $\bP$-a.e.\ $\omega \in \Omega$,
  	\begin{align*}
  		\sum_{w=0}^{x-1} A_w^2 = O(x) \quad \text{and} \quad \frac{1}{x} \sum_{w=0}^{x-1} s_x^2 \to \bE[ s_0^2 ],
  	\end{align*}
  	as $x \to \infty$, 
  	where $\bE[ s_0^2 ] > 0$. Consequently, the assumptions of Theorem \ref{thm:clt} are satisfied for 
  	$\bP$-a.e. $\omega \in \Omega$, and Theorem \ref{thm:clt_random}-(ii) follows.
  	\end{proof}

The proof of Theorem \ref{thm:llt_random} uses the following auxiliary result, which 
is a consequence of 
a maximal moment inequality 
due to Yang \cite{Y07}.

\begin{lem}\label{lem:yang_appli} Let $(\xi_n)_{n \ge 0}$ be a stationary $\alpha$-mixing sequence with
$$
\bE[ | \xi_0 |^q ] < \infty, \quad \text{$q = r + \delta$, $r > 2$, $\delta > 0$},
$$
and 
$$
\alpha(n) = O(n^{-\theta}), \quad \theta > r (r + \delta) / (2 \delta).
$$
Then, for any $u > 0$, as $k \to \infty$,
\begin{align}\label{eq:appli_yang}
	\max_{1 \le j \le u b(k)} \biggl| 
	\sum_{\ell = k + 1}^{k + j} ( \xi_\ell - \bE[ \xi_0 ] )
	\biggr| = o(k^p),
\end{align}
almost surely, where $b(k) = ( k \log k )^{1/2}$,
provided that
\begin{align}\label{eq:cond_p}
p > \frac{1}{r} + \frac{1}{4}.
\end{align}
\end{lem}

\begin{proof} Denote by $V_k$ the quantity on the left hand side of \eqref{eq:appli_yang}.
By the Borel--Cantelli lemma, it suffices to show that, for any $\ve > 0$,
\begin{align}\label{eq:bc_ineq}
	\sum_{k=1}^\infty \bP(  k^{ -p } V_k > \ve ) < \infty.
\end{align}
Let $S_j = \sum_{\ell = 0}^{j-1} ( \xi_\ell - \bE[ \xi_\ell ] )$.
For $k \ge 1$, by stationarity of $(\xi_n)$ and Markov's inequality,
\begin{align*}
	\bP(  k^{ -p } V_k > \ve ) =  \bP \biggl(  
	\max_{1 \le j \le u b(k)} |S_j| > \ve k^p
	\biggr) \le \ve^{-r}  k^{-pr} \bE \biggl[  \max_{1 \le j \le u b(k)} |S_j|^r   \biggr].
\end{align*}
Applying the maximal moment inequality for $\alpha$-mixing sequences from
\cite[Theorem 2.2]{Y07}, we obtain for arbitrary $\ve' > 0$
the upper bound 
\begin{align*}
	\bE \biggl[  \max_{1 \le j \le u b(k)} |S_j|^r   \biggr] \le K_{\ve',\delta, r} (
	[ k \log (k) ]^{ \frac{\ve'}{2} + \frac12  }
	+ [k \log(k)]^{ \frac{r}{4} }
	 ),
\end{align*}
where $K_{\ve',\delta, r} \in (0, \infty)$ is a constant. Consequently, \eqref{eq:bc_ineq} holds if 
\begin{align}\label{eq:cond_yang}
p > \frac{3}{2r} + \frac{\ve'}{2r}  \quad \text{and} \quad 
p > \frac{1}{r} + \frac{1}{4}.
\end{align}
Since $r > 2$, we have
$$
 \frac{3}{2r} + \frac{\ve'}{2r} < \frac{1}{r} + \frac{1}{4}
$$
whenever $\ve'$ is sufficiently small. Choosing such $\ve'$, 
\eqref{eq:cond_yang} 
reduces to
\eqref{eq:cond_p}, which completes the proof.
\end{proof}

\begin{proof}[Proof of Theorem \ref{thm:llt_random}] 
For $\omega \in \Omega$, we define $A_x, A_x'$ as in \eqref{eq:ax_random}. Further, we write
$$
B_x = 
\begin{cases}
	\bar{A}(\Theta^x \omega ), &\text{if $\bar{\beta}_* > 3$} \\
	\bar{A}'(\Theta^x \omega ) , &\text{if $\bar{\beta}_* \in (2,3]$}.
\end{cases}
$$
We will verify that (B1)-(B4) in Theorem \ref{thm:llt} are satisfied 
for $\bP$-a.e. $\omega \in \Omega$ with 
$\eta = 1/q$. Since we assume that
$$
q > \max \biggl\{   8,  \frac{2}{\beta_* - 2}   \biggr\},
$$
this choice of $\eta$ ensures that
$
0 \le \eta < \min \{1/2, (\beta_* - 2)/2 \},
$
and hence Theorem~\ref{thm:llt} applies.
This will imply Theorem~\ref{thm:llt_random}.

\noindent\textbf{(B1)}: First, recall from \eqref{eq:tail_tau} and \eqref{eq:rel_sx_ax} that
that $m_x \le C A_x$ and $s_x^2 \le C A_x$. As a consequence of (C1), 
the stationary random sequences $(m_x)$ and $(s_x^2)$ satisfy
$$
\bE[  m_0^q ] < \infty \quad \text{and} \quad \bE[  ( s_0^2 )^q ] < \infty.
$$
Note that, as $q > 8$, we have 
$$
\sum_{k=1}^\infty k^{ 1 / ( q/2 -  1  )   } \alpha(k) \le C \sum_{k=1}^\infty k^{ \frac{2}{q - 2}  - \frac{2q}{q-8}  } < \infty.
$$
It follows from the almost sure invariance principle for $\alpha$-mixing sequences 
in \cite[Theorem 2]{R95} combined with the 
moment condition in \cite[equation (1.4)]{R95}
that, for $\bP$-a.e. $\omega \in \Omega$,
\begin{align*}
	\frac{1}{x} \sum_{w=0}^{x-1} m_w = \mu + \theta_1(x)  \quad \text{and} \quad 
	\frac{1}{x} \sum_{w=0}^{x-1} s^2_w  = \sigma^2 + \theta_2(x)
\end{align*}
hold with $\theta_i(x) = O(x^{-1/2} \sqrt{ \log \log (x) } )$, 
$\mu = \bE[  m_0 ] > 1$ and $\sigma^2 = \bE[ s_0^2  ] > 0$.
In particular, (B1) holds for $\bP$-a.e. $\omega \in \Omega$
with $\eta = 1/q$.

\noindent\textbf{(B2)}: 
Since $\bE[ A_0^q ] < \infty$ by (C1) and \eqref{eq:rel_ax_ax'}, 
it follows from \cite[Lemma 4.1]{LS11} that 
$A_x = o(  x^{  1 / q } )$ for $\bP$-a.e. $\omega \in \Omega$. Moreover, 
it follows by (C1), (C2), and the strong law of large numbers in \cite[Corollary 1]{S93}
applied to $(B_x^4)_{x \ge 0}$ that, as $x \to \infty$,
$$
x^{-1} \sum_{w=0}^{x-1} B_w^{4} = O(1) \quad \text{for $\bP$-a.e. $\omega \in \Omega$.}
$$
Consequently, (B2) holds for $\bP$-a.e. $\omega \in \Omega$ with $\eta = 1/q$ and $\ve = 1$.

\noindent\textbf{(B3):}
Recall the definition of $L(x;u)$ in \eqref{eq:L_def} and that $\bE[ m_0^q ] < \infty$. 
Applying Lemma 
\ref{lem:yang_appli} with $p = 1/2 - \eta = 1/2 - 1/q$, we obtain
\begin{align*}
	L(x; u) = o(  x^{1/2-\eta}  ),
\end{align*}
for $\bP$-a.e. $\omega \in \Omega$, whenever
\begin{align*}
	\alpha(n) &= O( n^{-v} ) \quad \text{with $v > \frac{(q - \delta)q}{2\delta}$, $q - \delta > 2$,} \\
	\frac12 - \eta &> \frac{1}{q-\delta} + \frac14,
\end{align*}
where $\delta > 0$. This is equivalent with
\begin{align*}
	&\alpha(n) = O( n^{-v} ) \quad \text{with $v > \frac{(q - \delta)q}{2\delta}$,} \quad
	0 < \delta < \frac{q(q-8)}{q-4},
\end{align*}
which can be satisfied by choosing $\delta$ sufficiently close to the upper bound
$q(q - 8)/(q-4)$, provided that
\begin{align*}
\alpha(n) &= O( n^{-v} ) \quad \text{with $v > \frac{2q}{q-8}$}.
\end{align*}
Hence, (B3) holds for $\bP$-a.e. $\omega \in \Omega$ under (C2).

\noindent\textbf{(B4):} Under (C2) and (C3), \cite[Corollary 1]{S93} implies that $K_x = \bar{K}(\Theta^x \omega )$ satisfies
$$
\frac{1}{x} \sum_{w=0}^{x-1} K_w = O(1), \quad \text{as $x \to \infty$,}
$$
for $\bP$-a.e. $\omega \in \Omega$. Hence, (B4) is satisfied.
\end{proof}

	\bibliography{rwre}{}
	\bibliographystyle{plainurl}

\end{document}